\documentclass[reqno,a4paper]{amsart}
\usepackage[latin1]{inputenc}
\usepackage[english]{babel}
\usepackage{eucal,amsfonts,amssymb,amsmath,amsthm,epsfig,mathrsfs}
\usepackage{dsfont,cancel,soul}
\usepackage{color}
\usepackage{amscd,amsxtra}
\usepackage{enumerate}
\usepackage{latexsym}
\usepackage{bm}

\allowdisplaybreaks

\newcounter{ipotesi}

\Alph{ipotesi}
\makeatletter
\@namedef{subjclassname@2020}{%
  \textup{2020} Mathematics Subject Classification}
\makeatother
 \makeatletter \@addtoreset{equation}{section}

\makeatother \makeatletter

\newtheorem{theorem}{Theorem}[section]
\newtheorem{hyp}[theorem]{Hypotheses}{\rm}
{\rm}
\newtheorem{lemma}[theorem]{Lemma}

\newtheorem{proposition}[theorem]{Proposition}

\newtheorem{remark}[theorem]{Remark}{\rm}
\newtheorem{example}[theorem]{Example}

\newcounter{parentenv}

\newcommand{\R}{{\mathbb R}}

\newcommand{\N}{{\mathbb N}}

\newcommand{\Rd}{\mathbb R^d}

\newcommand{\Rm}{\mathbb R^m}

\newcommand{\T}{{\bm T}}

\newcommand{\f}{{\bm f}}
\newcommand{\bb}{{\bm b}}
\newcommand{\uu}{{\bm u}}
\newcommand{\A}{\bm{\mathcal A}}

\newcommand{\vv}{{\bm v}}

\newcommand{\q}{{\mathfrak q}}

\newcommand{\C}{\mathbb C}

\begin{document}

\title[Generation results for vector-valued elliptic operators]{Generation results for vector-valued elliptic operators with unbounded coefficients\\ in $L^p$ spaces}
\author[L. Angiuli, L. Lorenzi, E.M. Mangino and A. Rhandi]{Luciana Angiuli, Luca Lorenzi$^*$, Elisabetta M. Mangino, Abdelaziz Rhandi}\thanks{$^*$Corresponding author}
\address{L.A. \& E.M.M.:  Dipartimento di Matematica e Fisica ``Ennio De Giorgi'', Universit\`a del Salento, Via per Arnesano, I-73100 LECCE, Italy}
\address{L.L.: Dipartimento di Scienze Matematiche, Fisiche e Informatiche, Plesso di Mate\-matica, Universit\`a degli Studi di Parma, Parco Area delle Scienze 53/A, I-43124 PARMA, Italy}
\address{A.R.:  Dipartimento di Ingegneria dell'Informazione, Ingegneria Elettrica e Mate\-ma\-tica Applicata, Universit\`a degli Studi di Salerno, Via Giovanni Paolo II  132, I- 84084 FISCIANO (SA), Italy}
\email{luciana.angiuli@unisalento.it}
\email{luca.lorenzi@unipr.it}
\email{elisabetta.mangino@unisalento.it}
\email{arhandi@unisa.it}

\keywords{Vector-valued elliptic operators, unbounded coefficients, semigroups of bounded operators, improving summability properties}
\subjclass[2020]{35J47, 35K45, 47D06}

\begin{abstract}
We consider a class of vector-valued elliptic operators with unbounded coefficients, coupled up to the first-order, in the Lebesgue space $L^p(\Rd;\Rm)$
with $p \in (1,\infty)$. Sufficient conditions to prove generation results of an analytic $C_0$-semigroup $\T(t)$, together with a characterization of the domain of its generator, are given. Some results related to the hypercontractivity and the ultraboundedness of the semigroup are also established.
\end{abstract}

\maketitle

\section{Introduction}
\label{intro}
In this paper, we consider vector-valued elliptic operators with unbounded coefficients acting on smooth functions $\f:\Rd\to \R^m$ ($m \ge 2$) as follows:
\begin{equation}
\A\f= {\rm div}(Q\nabla \f)+\sum_{i=1}^d B^iD_i\f-V\f=:\A_0 \f+\sum_{i=1}^d B^iD_i\f-V\f,
\label{opA-intro}
\end{equation}
where $Q$, $B^i (i=1,\ldots,d)$ and $V$ are matrix-valued functions, and we study their realizations in the $L^p$-spaces with respect to the Lebesgue measure.
\par
In recent years, the interest on systems of elliptic and parabolic equations with unbounded coefficients has considerably grown motivated by a wide variety of mathematical models for physical and financial problems where they appear (backward-forward stochastic differential equations in connection with Nash equilibria in the theory of games, Navier Stokes equations, etc., see e.g., \cite{AALT,BGT,Dall,han-rha,Ha-He,hieber2,HRS}). Beside the analysis in spaces of bounded and continuous functions and the study of the so-called invariant measures (see e.g., \cite{AALT,AAL_Inv,AAL_Inv1,AL,DelLor11OnA}), the research on such systems of PDEs has been devoted to the $L^p$-setting with respect to the Lebesgue measure.
As it is known from the scalar case, the presence of an unbounded drift term produces additional difficulties to prove generation results of strongly continuous or analytic semigroups on the usual Lebesgue space $L^p(\Rd)$. Usually, one has to require strong conditions on the growth of the drift term or, as an alternative, to assume the existence of a dominating potential term.
The $L^p$-theory of parabolic systems with unbounded coefficients is not well developed and the literature
concerns essentially weakly coupled elliptic operators (i.e. the coupling between the equations is through a potential term), whose diffusion coefficients are
assumed to be uniformly elliptic and bounded. Moreover, often the techniques used to study these problems and prove generation results are based on perturbation methods, see \cite{HLPRS,KLMR,KMR,MR}.
The assumptions in \cite{HLPRS}  allow the drift term to grow like $|x|\log(1+|x|)$
and the potential term as $\log(1+|x|)$ as $|x|\to \infty$, and the generation result is proved via a Dore-Venni type theorem on sums of noncommuting operators,
due to Monniaux and Pr\"uss \cite{monniaux-pruss}.
This method works also in the case considered in \cite{KLMR} to prove generation results in $L^p$ spaces for vector-valued Schr\"odinger operators of the form $\A\uu={\rm div}(Q\nabla\uu)-V\uu$.
In that paper the operator is nondegenerate and the diffusion matrix $Q$ is symmetric, with entries which are bounded and continuously differentiable with bounded derivatives. The entries of the potential $V$ are locally Lipschitz continuous on $\Rd$ and satisfy the conditions $\langle V(x)\xi,\xi \rangle \ge|\xi|^2$ for every $x \in \Rd$, $\xi \in \Rm$,
and $|D_j V(-V)^{-\alpha}|\in L^\infty(\Rd)$ for some $\alpha \in [0,1/2)$.
The last assumption allows for potentials $V$ whose entries grow more than linearly at infinity. For instance the potential $V(x)= (1+|x|^r)V_0$, for every $x\in\Rd$,
where $V_0$ is an antisymmetric constant matrix and $r\in [1,2)$, is allowed. Under slightly different hypotheses on the potential $V$
(pointwise accretivity and local boundedness), generation results for the operator $\A$ as above are proved in \cite{KMR} but, differently from \cite{KLMR}, where the domain of the
$L^p$-realization of $\A$ is characterized as the intersection of the domains of diffusion and the potential terms of the operator, in \cite{KMR}
only a weak characterization of the domain is provided (in fact, the generation result is proved in the maximal domain of the realization of the operator $\A$ in $L^p(\Rd;\Rm)$).
A more general class of potentials, whose diagonal entries are polynomials of type $|x|^\alpha$ or even $|x|^r \log(1+|x|)$ as well as $e^{|x|}$, for $\alpha,r\ge 1$, is considered in \cite{MR} where the operator $\A$ is perturbed by a scalar potential $v \in W^{1,\infty}_{\rm loc}(\Rd)$ satisfying $|\nabla v|\le cv$ for some positive constant $c$. A perturbation theorem (due to Okazawa \cite{Oka} and used in \cite{MR}) works for a matrix-valued perturbation of $V$ in the $L^2$-setting (see \cite{AngLorMan}) allowing for different growth rates of the type of \cite{MR} for the diagonal entries of the potential matrix. In \cite{AngLorMan} the operator $\A$ is also perturbed by a diagonal first-order term that can grow at most linearly at infinity.

In this paper, using direct methods and suitable assumptions that depend on $p\in (1,\infty)$, we prove that the realization of $\A$ in $L^p(\Rd;\Rm)$ with domain $D_p=\{\uu\in L^p(\Rd;\Rm)\cap W^{2,p}_{\rm loc}(\Rd;\Rm): {\rm div}(Q\nabla \uu),$ $V\uu\in L^p(\Rd;\Rm)\}$ generates an analytic $C_0$-semigroup on $L^p(\Rd;\Rm)$. Our results improve all the above mentioned results: the novelty of this paper relies on the form of the operator,
where a coupling first-order term is allowed, and on the fact that the diffusion coefficients can be unbounded in $\Rd$.
Actually systems of elliptic operators coupled at the first-order are considered also in \cite{AngLorPal} with a different approach. More precisely, in such a paper conditions to extrapolate the semigroup $\T(t)$, first generated in $C_b(\Rd;\Rm)$ in \cite{AALT}, to the $L^p$-scale are provided, but with no characterization of the domain.
In general, the determination of the domain is a quite complicate issue which requires more assumptions on the coefficients and which can be simplified considerably
 assuming that the diffusion coefficients are uniformly elliptic and bounded. Instead, in the case of unbounded diffusion coefficients, already in the scalar case there are only partial results (see \cite{mpps} and the reference therein).

Our main assumptions are listed in Hypotheses \ref{hyp_0} and are inspired by those considered in \cite{mpps} where the scalar equation is studied.
We introduce an auxiliary potential $v$ that controls the matrix-valued function $V$ and, through Hypothesis \ref{hyp_0}(iv), allows to prove the regular $L^p$-dissipativity of the operator $\A$.
On the other hand, the assumption \eqref{drift-cont} and the oscillation condition in Hypothesis \ref{hyp_0}(v) are crucial to interpolate the term $\sum_{i=1}^d B_i D_i\uu$ between $\A_0\uu$ and $V\uu$ (see Remark \ref{no-small}) and consequently, assuming further \eqref{pdis}, if $p\in (1,2)$, and \eqref{pdis-1}, if $p \in [2,\infty)$, to identify the domain. Condition \eqref{cre-v} together with a bound on $\gamma$, was already used in the scalar case in \cite{davies_some,davies_some1}  to show that the domain of the Schr\"odinger  operator $-\Delta+V$ in $L^2(\Rd)$ coincides with
$W^{2,2}(\Rd)\cap\{u\in L^2(\Rd): Vu\in L^2(\Rd)\}$ both for smooth and singular potentials. There are counterexamples which show that this domain characterization  fails if \eqref{cre-v} is true with a too large $\gamma$ (see \cite[Example 3.7]{metafune}).  We point out that, already in the case of Schr\"odinger vector-valued operators, our hypotheses allow for the entries of $V$ to grow at infinity as $e^{|x|^\beta}$, for every $\beta>0$, improving the growth-rate considered in \cite{MR}, (see Example \ref{example-1}).
In addition, our results, besides giving a precise description on the domain of the generator in $L^p(\Rd;\Rm)$,
work under less restrictive assumptions than those considered in \cite{AngLorPal}.
In that paper two different set of hypotheses are considered: the first one imposes a sign on the drift term, the second one forces the matrices $B^i$ to be bounded when the diffusion coefficients are themselves bounded. In our case, neither a sign on the drift term is assigned, nor the drift has to be bounded when $Q$ is bounded. Moreover, to extrapolate the semigroup $\T(t)$ to the $L^p$-scale ($p \in [2,\infty)$), the first set of assumptions of \cite{AngLorPal} forces the quadratic form associated with the matrix-valued function $-2V-\sum_iD_iB^i$ to be bounded from above. Our hypotheses cover also cases in which the previous form is not bounded from above.

We then slightly change our main hypotheses, to make them independent of $p$ in the range $[p_0,\infty)$, for some $p_0>1$ (see Remark \ref{forallp}).
Under, this new set of assumptions, we prove the consistency of the semigroups $\T_p(t):=\T(t)$ for every $p\ge p_0$. We then show that each operator $\T(t)$ maps $L^p(\Rd;\Rm)$ into $L^q(\Rd;\Rm)$ for $p_0\le p\le q\le \infty$ (see Theorem \ref{thm_hyp}). This is done through a comparison between the semigroup $\T_{\infty}(t)$ generated by $\A$ in $C_b(\Rd;\Rm)$ (that actually coincides with that generated in $L^p(\Rd;\Rm)$) and the scalar semigroup associated to the Schr\"odinger operator $\mathcal A_v= {\rm div}(Q\nabla)-v$.
The analysis in $L^1(\Rd;\Rm)$ will be deferred to a future paper.
\medskip

\noindent
{\bf Notation.}
Let $d, m\in\N$ and let $\mathbb K=\mathbb R$ or $\mathbb K=\mathbb C$. We denote  by $\langle\cdot, \cdot\rangle$ and by $|\cdot|$, respectively, the Euclidean inner product and the norm in $\mathbb K^m$.
Vector-valued functions are displayed in bold style. Given a function $\uu: \Omega \subseteq \Rd  \rightarrow {\mathbb K}^m$, we denote by $u_k$ its $k$-th component. For every $p\in [1,\infty)$,  $L^p(\Rd, \mathbb K^m)$  denotes the classical vector valued Lebesgue space endowed with the norm
$\|\f\|_p=(\int_{\Rd} |\f(x)|^pdx)^{1/p}$.
The canonical pairing between $L^p(\Rd, \mathbb K^m)$ and $L^{p'}(\Rd, \mathbb K^m)$ ($p'$ being the index conjugate to $p$), i.e., the integral over $\Rd$ of the function
$x\mapsto \langle \uu(x), \vv(x)\rangle$ when $\uu\in L^p(\Rd, \mathbb K^m)$ and $\vv\in L^{p'}(\Rd, \mathbb K^m)$, is denoted by $\langle \uu,\vv\rangle_{p,p'}$.
For $k\in\mathbb N$, $W^{k,p}(\Rd, \mathbb K^m)$ is the classical vector valued Sobolev space, i.e., the space of all functions $\uu\in L^p(\Rd, \mathbb K^m)$ whose components have distributional derivatives up to the order $k$, which belong to $L^p(\R^d, \mathbb K)$. The norm of $W^{k,p}(\Rd, \mathbb K^m)$ is
denoted by $\|\cdot\|_{k,p}$.
If the matrices $B^i$ ($i=1,\ldots,d$) have differentiable entries, then we set ${\rm div}B(x)=\sum_{i=1}^dD_iB^i(x)$ for every $x\in\Rd$, where $D_iB^i$ is the matrix whose entries are obtained differentiating with respect to the variable $x_i$ the corresponding entries of the matrix $B^i$.
Given a $d\times d$-matrix-valued function $Q$, we denote by $\q(u,v)$ the function defined by $x\mapsto\langle Q(x)\nabla u(x),\nabla v(x)\rangle$ on
smooth enough functions $u$ and $v$. We simply write $\q(u)$ when $u=v$. Finally, given a vector-valued function $\uu$ and $\varepsilon>0$, we denote by $w_{\varepsilon}$ the scalar valued function $w_{\varepsilon}=(|\uu|^2+\varepsilon)^{1/2}$.

\section{Cores }\label{Hyp-sec}

The aim of this section is to prove vector valued versions of results about cores for elliptic operators with unbounded coefficients, in the line of those proved in \cite{ALM}. Throughout the section,  we will consider the elliptic operator $\A$ in \eqref{opA-intro} assuming that the matrix $B^i$ are diagonal, i.e., $B^i=b_iI$ for some functions $b_i:\Rd\to\R$, and we set ${\bm b}=(b_1,\ldots,b_d)$.

In the following lemma we adapt some known results  about scalar elliptic regularity  to the vector valued case.

\begin{lemma}\label{reg}
Suppose that $Q$ is locally positive definite, i.e.
$ \langle Q (x) \xi,\xi\rangle
 \geq \mu(x) |\xi|^2$ for every $x,\xi \in \Rd$ and some positive function $\mu$ such that $\inf_K\mu>0$ for every compact set $K\subset \Rd$.
Further suppose that $\uu\in L^1_{\rm loc}(\Rd;\Rm)$, $\bm{f}\in L^r_{\mathrm{loc}}(\Rd;\Rm)$  satisfy the variational formula
\begin{equation}\label{star}
\int_{\Rd}\uu\A\bm{\varphi}dx = \int_{\Rd}\f\bm{\varphi}dx, \qquad\;\,\bm{\varphi}\in C_c^\infty(\Rd;\Rm),
\end{equation}
for some $r\in (1,\infty)$. Then, the following properties hold true.
\begin{enumerate}[{\rm (i)}]
\item
If $r>d\geq 2$ and $q_{ij} \in W^{1,r}_{\mathrm{loc}}(\Rd)$, $b_j, v_{hh} \in   L^r_{\mathrm{loc}}(\Rd)$ for $i,j=1, \dots, d$, $h=1, \dots, m$, and $v_{hk}\in L^\infty_{\mathrm{loc}}(\Rd)$ for $h,k=1, \dots, m$, $h\not=k$, then $\uu\in W^{1,r}_{\mathrm{loc}}(\Rd;\Rm)$.
\item
If $q_{ij}, b_i \in C^1(\Rd)$ for $i,j=1, \dots, d$, $v_{hh} \in   L^p_{\mathrm{loc}}(\Rd)$ for $h=1,\ldots,m$ and $v_{hk}\in L^\infty_{\mathrm{loc}}(\Rd)$ for $h,k=1, \dots, m$, with $h\neq k$, then $\uu\in W^{2,p}_{\mathrm{loc}}(\Rd;\Rm)$.
\end{enumerate}
\end{lemma}

\begin{proof}
Let ${\bm e}_1,\ldots,{\bm e}_m$ denote the canonical basis of $\Rm$. Writing \eqref{star} with $\bm{\varphi}=\varphi {\bf e}_h$ for some $\varphi\in C_c^\infty(\Rd)$ and $h\in\{1, \dots, m\}$ we get
\begin{eqnarray*}
\int_{\Rd}\left({\rm div}(Q\nabla \varphi)+ \langle \bb,  \nabla \varphi\rangle -v_{hh} \varphi \right) u_h  dx  = \int_{\Rd}\bigg (f_{h}+\sum_{k\neq h} v_{kh}u_k\bigg )\varphi dx.
\end{eqnarray*}
By the arbitrariness of $h$ we get both assertions from  the scalar case, by applying \cite[Corollary 2.10]{BKR} and standard elliptic regularity (see \cite{A}, or, e.g., \cite[Theorem C.1.3]{newbook}).
\end{proof}
\begin{theorem}\label{core}
Let $p\in (1,\infty)$ and assume that, for some $r>d\geq 2$,
$q_{ij}, b_j \in W^{1,r}_{\mathrm{loc}}(\Rd)$ for every $i,j=1, \dots, d$, $v_{hh} \in   L^r_{\mathrm{loc}}(\Rd)$ for every $h=1,\ldots,m$ and $v_{hk}\in L^\infty_{\mathrm{loc}}(\Rd)$ for every $h,k=1, \dots, m$, with $h\neq k$. Further, assume that $Q$ is locally uniformly elliptic and there exists a positive function $\psi \in C^1(\Rd)$ such that $\lim_{|x|\to\infty}\psi(x)=\infty$ and
\begin{equation}\label{beta1}
\frac{\langle\bb,\nabla \psi\rangle}{\psi\log \psi}\geq  -C_1\ \  {and}\ \ \frac{\langle Q\nabla \psi, \nabla \psi\rangle}{(\psi\log \psi)^2} \leq C_2
\end{equation}
for some positive constants $C_1, C_2$ and that
\begin{equation}\label{etoile-0}
p^{-1}(\mathrm{div}\bb)|\xi|^2 + \langle V \xi, \xi\rangle \geq 0
\end{equation}
in $\Rd$ for every $\xi\in \Rm$.
Then, the operator $(\A, C_c^{\infty}(\Rd;\Rm))$ is closable on $L^p(\Rd;\Rm)$ and its closure generates a
strongly continuous semigroup.
\end{theorem}

\begin{proof}
First, we observe that, thanks to \eqref{etoile-0}, $(\A, C_c^{\infty}(\Rd;\Rm))$ is dissipative. Hence, by \cite[Proposition 3.14]{engnagel} we deduce that $(\A, C_c^{\infty}(\Rd;\Rm))$ is closable on $L^p(\Rd;\Rm)$. Then, we have only to show that $(\lambda I- \A)(C_c^\infty(\Rd;\Rm))$ is dense in $L^p(\Rd;\Rm)$ for some $\lambda >0$.
Since $\psi(x)$ tends to $\infty$ as $|x|\to\infty$, without loss of generality we can assume  that $\psi(x)\ge 1$ for every $x\in\Rd$.

Fix $\lambda >0$ and let $\uu\in L^{p'}(\Rd;\Rm)$ be such that $\langle \lambda \bm{\varphi} - \A\bm{\varphi}, \uu\rangle_{p,p'}=0$ for every $\bm{\varphi}\in C_c^\infty(\Rd;\Rm)$.

By Lemma \ref{reg}, $\uu\in W^{1,r}_{\mathrm{loc}}(\Rd;\Rm)$ and therefore, for every $\bm{\varphi} \in C_c^\infty(\Rd;\Rm)$,
\begin{align}\label{parti1}
\lambda \int_{\Rd} \langle \bm{\varphi} , \uu \rangle dx = \int_{\Rd}\bigg(\sum_{i=1}^d [b_i\langle D_i\bm{\varphi}, \uu\rangle-\mathfrak{q}(\varphi_i,u_i)]-\langle V\bm{\varphi},\uu\rangle \bigg)dx.
\end{align}

The equality \eqref{parti1} extends by density to every function $\bm{\varphi} \in W^{1,r'}(\Rd;\Rm)$ with compact support. In particular, since $r'< 2$, formula \eqref{parti1}  holds true for every $\bm{\varphi} \in W^{1,2}_{\mathrm{loc}}(\Rd;\Rm)$ with compact support.

Fix a smooth decreasing function $\zeta:[0,\infty)\to [0,1]$ such that $\zeta(s)=1$ if $s\in [0,1]$ and $\zeta(s)=0$ if $s\geq 2$, and set $\zeta_n=\zeta(n^{-1}\log \psi)$. Clearly $\zeta_n\in C_c^\infty(\Rd;\Rm)$ and $\lim_{n\to\infty}\zeta_n(x)=1$ for every $x\in\Rd$.

Since $\uu\in W^{1,r}_{\mathrm{loc}}(\Rd;\Rm)$ with $r>d$, it follows from the Sobolev embedding theorem that $\bm{\varphi}_n= \zeta_n^2 \uu w_{\varepsilon}^{p'-2}$ belongs to $W^{1,2}_{\mathrm{loc}}(\Rd;\Rm)$
and has compact support. Hence, writing \eqref{parti1} with $\bm\varphi=\bm\varphi_n$, we get
\begin{align}
\lambda \int_{\Rd} |\uu|^2w_{\varepsilon}^{p'-2}\zeta_n^2 dx
=&-\sum_{i=1}^m\int_{\Rd}{\mathfrak q}(\zeta_n^2u_i w_{\varepsilon}^{p'-2},u_i)dx
-\int_{\Rd}\langle V\uu,\uu\rangle \zeta_n^2w_{\varepsilon}^{p'-2}dx\notag\\
&+\int_{\Rd} \sum_{i=1}^d\sum_{j=1}^m b_i D_i (\zeta_n^2u_j w_{\varepsilon}^{p'-2} )u_j dx\notag\\
=&\!:{\mathscr I}_1-\int_{\Rd}\langle V\uu, \uu\rangle\zeta_n^2w_{\varepsilon}^{p'-2}dx+{\mathscr I}_2.
\label{dom}
\end{align}
Taking into account that $D_iw_{\varepsilon}=w_{\varepsilon}^{-1}\langle D_i\uu,\uu\rangle$ for every $i=1,\ldots,d$, it is easy to check that
\begin{align}
{\mathscr I}_1
=&-\sum_{i=1}^m \int_{\Rd} \q(u_i)\zeta_n^2 w_{\varepsilon}^{p'-2}dx-\int_{\Rd}\q(\zeta_n^2,w_{\varepsilon})w_{\varepsilon}^{p'-1}dx\notag\\
&-(p'-2)\int_{\Rd} \q(w_{\varepsilon})\zeta_n^2 w_{\varepsilon}^{p'-2}dx.
\label{I1-dom}
\end{align}
Moreover, applying H\"older's inequality we can estimate
\begin{align}
\bigg |\int_{\Rd} \q(\zeta_n^2,w_{\varepsilon})w_{\varepsilon}^{p'-1}dx\bigg |
\le \delta\int_{\Rd}  \q(w_{\varepsilon})w_{\varepsilon}^{p'-2}\zeta_n^2dx+
\frac{1}{\delta} \int_{ \Rd} \q(\zeta_n)w_{\varepsilon}^{p'} dx
\label{I1-dom-1}
\end{align}
for every $\delta>0$. Finally, observing that
\begin{eqnarray*}
\sum_{j=1}^m\nabla (\zeta_n^2u_jw_{\varepsilon}^{p'-2})u_j=\nabla (\zeta_n^2|\uu|^2w_{\varepsilon}^{p'-2})-\frac{1}{p'}\nabla(\zeta_n^2w_{\varepsilon}^{p'})+\frac{1}{p'}w_{\varepsilon}^{p'}\nabla (\zeta_n^2)
\end{eqnarray*}
and integrating by parts we obtain
\begin{align}
{\mathscr I}_2=& - \frac{1}{ p}\int_{\Rd}({\rm div}\bb)|\uu|^2 \zeta_n^2w_{\varepsilon}^{p'-2}dx+\frac{\varepsilon}{p'}\int_{\Rd}({\rm div}\bb) \zeta_n^2w_{\varepsilon}^{p'-2}dx\notag\\
&+\frac{1}{p'} \int_{\Rd} \langle \bb, \nabla \zeta_n^2\rangle w_{\varepsilon}^{p'}dx.
\label{I2-dom}
\end{align}

Replacing \eqref{I1-dom} and \eqref{I2-dom} in formula \eqref{dom} and taking \eqref{I1-dom-1} (with $\delta<p'-2$) and \eqref{etoile-0} into account, we get
\begin{align}
\lambda \int_{\Rd} |\uu|^2w_{\varepsilon}^{p'-2}\zeta_n^2 dx
\le &\frac{1}{\delta}\int_{\Rd} \q(\zeta_n)w_{\varepsilon}^{p'} dx+\frac{\varepsilon}{p'}\int_{\Rd}({\rm div}\bb) \zeta_n^2w_{\varepsilon}^{p'-2}dx\notag\\
&+\frac{1}{p'}\int_{\Rd}\langle \bb, \nabla \zeta_n^2 \rangle w_{\varepsilon}^{p'}dx\notag\\
\leq& \frac{1}{\delta n^2} \int_{\Rd} \frac{\q(\psi)}{\psi^2}[\zeta'(n^{-1} \log\psi)]^2 w_{\varepsilon}^{p'} dx
+\frac{\varepsilon}{p'}\int_{\Rd}({\rm div}\bb) \zeta_n^2w_{\varepsilon}^{p'-2}dx\notag\\
&+\frac{2}{p'n}\int_{\Rd} \frac{\langle \bb, \nabla \psi\rangle}{\psi} \zeta'(n^{-1} \log\psi) \zeta_n w_{\varepsilon}^{p'}dx.
\label{rovato}
\end{align}

Note that the second to last integral in \eqref{rovato} converges to $0$ as $\varepsilon \to 0$. Indeed, for each $\varepsilon \in (0,1)$ we can estimate
\begin{align*}
\varepsilon \int_{\Rd}({\rm div}\bb) \zeta_n^2w_{\varepsilon}^{p'-2}dx
\le \varepsilon^{\frac{p'\wedge 2}{2}}\int_{\Rd}|{\rm div}\bb|\,w_1^{(p'-2)^+} \zeta_n^2dx,
\end{align*}
which vanishes as $\varepsilon \to 0^+$ since the function $\zeta_n$ is compactly supported in $\Rd$ and the functions ${\rm div} \bb$, $w_1$ are, respectively, locally integrable and locally bounded on $\Rd$.
Here, $(p'-2)^+$ denotes the positive part of $p'-2$.
Hence, letting $\varepsilon \to 0$ and using the dominated convergence theorem, we deduce that
\begin{align}
\lambda \int_{\Rd} |\uu|^{p'}\zeta_n^2 dx
\leq& \frac{1}{\delta n^2} \int_{\Rd} \frac{\q(\psi)}{\psi^2}[\zeta'(n^{-1} \log\psi)]^2 |\uu|^{p'} dx\notag\\
&+\frac{2}{p'n}\int_{\Rd} \frac{\langle \bb, \nabla \psi\rangle}{\psi} \zeta'(n^{-1} \log\psi) \zeta_n |\uu|^{p'}dx.
\label{rovato1}
\end{align}

Since, for every $n\in\N$, $\zeta'(n^{-1}\log \psi(x))\neq 0$ only if $ 1\leq n^{-1}\log \psi(x) \leq 2$, taking \eqref{beta1} into account we can estimate
\begin{align*}
\bigg |\frac{1}{\delta n^2} \frac{\q(\psi)}{\psi^2}[\zeta'(n^{-1} \log\psi)]^2|\uu|^{p'} \bigg |\le \frac{4\q(\psi)}{\delta\psi^2 \log^2 \psi} \|\zeta'\|_{\infty}^2|\uu|^{p'}
\le 4C_2\delta^{-1}\|\zeta'\|_{\infty}^2|\uu|^{p'}.
\end{align*}
Moreover, since $\zeta'\leq 0$ on $[0,\infty)$, it follows that
\begin{eqnarray*}
\frac{1}{p'n}\int_{\Rd}\frac{\langle \bb, \nabla \psi\rangle}{\psi} \zeta'(n^{-1} \log\psi) \zeta_n|\uu|^{p'}dx \leq\frac{2C_1}{p'}\|\zeta'\|_{\infty}\int_{\Rd}|\uu|^{p'}dx.
\end{eqnarray*}
Hence, by dominated convergence we can let $n$ tend to $\infty$ in both sides of \eqref{rovato1} and conclude that
$\lambda\|\uu\|_{p'}\leq 0$, whence $\uu={\bm 0}$.
\end{proof}

\begin{theorem}
\label{thm-1}
Let $p\in (1,\infty)$ and assume that
$q_{ij}, b_i \in C^1(\Rd)$ for every $i,j=1,\ldots,d$, $v_{hh} \in   L^p_{\mathrm{loc}}(\Rd)$ for every $h=1,\ldots,m$ and $v_{hk}\in L^\infty_{\mathrm{loc}}(\Rd)$ for every $h,k=1, \dots, m$ with $h\not=k$. Further, assume that
$Q$ is locally uniformly elliptic and that there exists a positive function $\psi \in C^1(\Rd)$, which diverges to $\infty$ as $|x|$ tends to $\infty$, such that
\begin{eqnarray*}
\left|\frac{\langle\bb,\nabla \psi\rangle}{\psi\log \psi}\right|\leq  C_1\;\,  \mathrm {and}\;\, \frac{\langle Q\nabla \psi,\nabla \psi\rangle}{(\psi\log \psi)^2}\le C_2
\end{eqnarray*}
for some constants $C_1,C_2>0$. Finally, assume that condition \eqref{etoile-0} holds true.
Then, the realization $\bm A_p$ of the operator $\A$ in $L^p(\Rd,\R^m)$, with domain
$D_{p,{\rm max}}=\{ \uu\in L^p(\Rd;\Rm)\cap W_{\mathrm{loc}}^{2,p}(\Rd;\Rm): \A\uu\in L^p(\Rd;\Rm)\}$,
generates a contraction semigroup in $L^p(\Rd;\Rm)$. Moreover, the space $C_c^\infty(\Rd;\Rm)$ is a core for $(\A, D_{p,{\rm max}})$.
Finally, if ${\bm b}$ identically vanishes on $\Rd$, then, the previous semigroups exist for every $p\in (1,\infty)$ and are consistent.
\end{theorem}

\begin{proof}
Let $(\overline\A, D)$ be the  closure of $(\A, C_c^\infty(\Rd;\Rm))$ in  $L^p(\Rd;\Rm)$ and fix $\uu\in D$.
Then, there exists a sequence $(\uu_n)$ in $C_c^\infty(\Rd;\Rm)$ such that
$\uu_n$ and $\A\uu_n$ converge, respectively to some function $\uu$ and $\bm{g}$ in $L^p(\Rd;\Rm)$, as $n$ tends to $\infty$.
Hence, for every $\bm{\varphi}\in C_c^\infty(\Rd;\Rm)$ it follows that
\begin{eqnarray*}
\int_{\Rd} \langle \uu, \A^* \bm{\varphi}\rangle dx =  \int_{\Rd}\langle \bm{g}, \bm{\varphi} \rangle dx,
\end{eqnarray*}
where $\A^*$ is the formal adjoint to the operator $\A$, which implies that $\overline \A \uu=\ \bm{g} =\A\uu$ distributionally. By standard elliptic regularity results,  we deduce that $\uu\in D_{p,\max}$.

Let us now prove that $\lambda I- \A$ is injective on $D_{p,{\rm max}}$ for some $\lambda >0$. For this purpose, we fix $\uu\in D_{p,\max}$ such that $\lambda \uu=\A\uu$. Then, for every $\bm{\varphi}\in C_c^\infty(\Rd;\Rm)$,
it holds that
$\langle \uu, \lambda \bm{\varphi} - \A^*  \bm{\varphi}  \rangle_{p,p'} = \langle \lambda \uu -\A \uu, \bm{\varphi}\rangle_{p,p'}=0$.
Since $C_c^\infty(\Rd;\Rm)$ is a core for $\A^*$ in $L^{p'}(\Rd;\Rm)$, due to Theorem \ref{core}, we conclude that $\uu={\bf 0}$.

Next, we fix a function $\uu \in D_{p,\max}$ and set $\bm{v}=\lambda\uu- \A\uu$.   By Theorem \ref{core}, ${\bm v}= \lambda \bm{w} - \overline{\A} \bm{w} = \lambda \bm{w} - \A \bm{w} $ for some $\bm{w}\in D\subseteq D_{p.\max}$. By the injectivity of $\A$ on $D_{p,\max}$ we get that $\uu={\bm w} \in D$.

Finally, let us assume that ${\bm b}$ identically vanishes on $\Rd$. To prove that the semigroups generated by the operators ${\bm A}_p$ and ${\bm A}_q$
are consistent, one can take advantage of the Trotter product formula (see \cite[Corollary III.5.8]{engnagel}) to write
\begin{eqnarray*}
e^{t{\bm A}_r}\f=\lim_{n\to\infty}\big (e^{\frac{t}{n}{\bm A}_r^0}e^{-\frac{t}{n}{\bm V}_r}\big )^n\f,\qquad\;\,\f\in L^r(\Rd;\R^m),
\end{eqnarray*}
for every $t>0$ and $r\in\{p,q\}$ if $p,q\ge 2$, where $e^{-t{\bm V}_r}$ is the strongly semigroup in $L^r(\Rd;\Rm)$ generated by
the multiplication operator $\uu\mapsto -{\bm V}_r\uu$, with $D({\bm V}_r)=\{\uu\in L^r(\Rd;\R^m): V\uu\in L^r(\Rd;\R^m)\}$ and
$e^{t{\bm A}_r^0}\f=(e^{tA_r^0}f_1,\ldots,e^{tA_r^0}f_m)$ for every $t>0$ and $\f\in L^r(\Rd;\R^m)$, where $e^{tA_r^0}$ is the scalar semigroup generated
by the operator ${\rm div}(Q\nabla )$. Both the semigroups $e^{tA_r^0}$ and $e^{-t{\bm V}_r}$ are consistent on the $L^p$-scale.
If $p,q\in (1,2]$, then we observe that the operator $\A^*$ adjoint to $\A$ satisfies the same assumptions as the operator $\A$.
Therefore, for $r\in\{p,q\}$, the semigroup $e^{t{\bm A}_r}$ is the adjoint of the semigroup generated in $L^{r'}(\Rd;\R^m)$ by the closure of the operator $(\A^*,C^{\infty}_c(\Rd;\R^m))$. Denote by $\T_{r'}(t)$ this semigroup. Then, for every $t>0$, $\f\in L^p(\Rd;\R^m)\cap L^q(\Rd;\R^m)$ and $\bm\varphi\in C^{\infty}_c(\Rd;\R^m)$, we can write
\begin{eqnarray*}
\langle e^{t{\bm A}_p}\f,\bm{\varphi}\rangle_{p,p'}=\langle\f,\T_{p'}(t)\bm{\varphi}\rangle_{p,p'}
=\langle\f,\T_{q'}(t)\bm{\varphi}\rangle_{q,q'}
=\langle e^{t{\bm A}_q}\f,\bm{\varphi}\rangle_{q,q'}
\end{eqnarray*}
and the equality $e^{t{\bm A}_p}\f=e^{t{\bm A}_q}\f$ follows. Finally, if $p<2$ and $q>2$, then $e^{t{\bm A}_p}\f=e^{t{\bm A}_2}\f=e^{t{\bm A}_q}\f$ for every $t>0$ and $\f\in C^{\infty}_c(\Rd;\R^m)$. Approximating $\f\in L^p(\Rd;\R^m)\cap L^q(\Rd;\R^m)$ in $L^p(\Rd;\R^m)\cap L^q(\Rd;\R^m)$ with a sequence of functions of function $(\f_n)\subset C^{\infty}_c(\Rd;\R^m)$, the equality $e^{t{\bm A}_p}=e^{t{\bm A}_q}\f$ on $L^p(\Rd;\R^m)\cap L^q(\Rd;\R^m)$ follows for every $t>0$ also in this case.
\end{proof}

\section{The full operator $\A$}

In this section, we consider the elliptic operator $\A$ defined in \eqref{opA-intro} assuming that $p \in (1,\infty)$ and that the coefficients $Q=(q_{ij})$, $B^i=(B^i_{hk})$, $V=(v_{hk})$ satisfy the following assumptions:
\begin{hyp}\label{hyp_0}
\begin{enumerate}[\rm(i)]
\item
$Q$ is real, symmetric matrix with entries in $C^1 (\Rd)$. Moreover,
$\langle Q(x)\xi, \xi\rangle >0$ for every $x\in\Rd$ and $\xi\in\Rd\setminus\{0\}$;
\item
$B^i: \Rd \rightarrow \R^{m\times m}$ are symmetric matrix-valued functions, with coefficients of class $C^1$ over $\Rd$ for every $i=1, \dots, d$;
\item
$V:\Rd \rightarrow \R^{m\times m}$ is a measurable matrix-valued function;
\item
there exist a function $v\in C^1(\Rd)$ with positive infimum $c_0$ and positive constants  $\kappa, c_1$ and $\theta<p$ such that
\begin{align}
&\langle V(x)\xi, \xi \rangle\ge v(x)|\xi|^2, \qquad  |V(x)\xi| \le c_1 v(x)|\xi|,\notag \\
& \sum_{h,k=1}^m \left\vert \sum_{i=1}^d B^i_{hk}(x) \eta^{k}_i \right\vert \leq \kappa \sqrt{v(x)}\sum_{k=1}^m\langle Q(x)\eta^k,\eta^k \rangle^\frac{1}{2},
\label{drift-cont}\\
&\langle {\rm div} B(x) \xi, \xi\rangle\geq -\theta v(x)|\xi|^2\notag
\end{align}
for every $x\in \Rd$, $\xi \in \Rm$, $\eta^k=(\eta^k_1, \ldots, \eta^k_d)\in \Rd$ and $k=1, \dots, m$;
\item
there exist positive constants $\gamma$, $C_\gamma$ and $c_2$ such that
\begin{eqnarray}
 &   \mbox{ if }\ &1<p\leq 2:\ \qquad  \langle Q(x)\nabla v(x), \nabla v(x)\rangle^{\frac 1 2} \le \gamma v(x)^{3/2}+ C_\gamma
 \label{cre-v}\\[2mm]
 & \mbox{ if }\ & p> 2:\ \qquad\qquad
\left\{
\begin{array}{l}
|Q(x)|^{\frac{1}{2}} |\nabla v(x)| \le \gamma v(x)^{3/2}+ C_\gamma,\\[1mm]
|Q(x)|\leq c_2 v(x),\\[1mm]
\sup_{|x-y|\leq \rho(x)}|\nabla Q(y)| \leq c_2 \mu(x) \rho(x)^{-1}
\end{array}
\right.
\notag
\end{eqnarray}
for every $x\in\Rd$, where $\mu(x)$ is the minimum eigenvalue of $Q(x)$ and $\rho(x)= |Q(x)|^{\frac{1}{2}} v(x)^{-\frac{1}{2}}$;
\item
there exists a positive function $\psi \in C^1(\Rd)$ such that $\lim_{|x|\to\infty}\psi(x)=\infty$ and
$\langle Q\nabla \psi, \nabla \psi\rangle \leq C \psi^2\log^2 \psi$.
\end{enumerate}
\end{hyp}

\begin{proposition}
Under Hypotheses $\ref{hyp_0}(i)$-$(vi)$, assume further that $\gamma^2\leq \frac{4}{p-1}$, if $1<p<2$. Then, for every $\varepsilon>0$ there exists a positive constant $K_\varepsilon$, depending on $p$ and the constants $\kappa, c_0,c_1,  c_2, \gamma, C_\gamma$, such that
\begin{equation}\label{interpolazione}
 \bigg\| \sum_{i=1}^d B^i  D_i\uu\bigg\|_p\leq \varepsilon \|\A_0\uu\|_p+ K_\varepsilon \|V\uu\|_p,\qquad\;\,\uu\in C^{\infty}_c(\Rd;\Rm).
\end{equation}
\end{proposition}

The proof follows from the analogous scalar inequality proved in \cite[Lemma 2.5, Proposition 3.3]{mpps}.
This is the unique point of the paper where we need conditions on the oscillation of the function $Q$, in the case $p>2$.

\begin{remark}
\label{no-small}
{\rm It is worth observing explicitly that, as already proved in the scalar case, condition \eqref{interpolazione} does not imply that the drift term is a small perturbation of $\A_0$ or $\A_0-V$, see \cite[Remark 3.6]{metafune}}.
\end{remark}

For every $p\in (1,\infty)$, let ${\bm A}_p$ denote the realization of the operator $\A$ in $L^p(\Rd;\Rm)$ with domain
\begin{align*}
D_p=&\{\uu\in W^{2,p}_{{\rm loc}}(\Rd;\mathbb C^m):  \uu, \A_0\uu, V\uu\in L^p(\Rd;\mathbb C^m)\}\\
=& \{\uu\in W^{2,p}_{{\rm loc}}(\Rd;\mathbb C^m):  \A_0\uu, v\uu\in L^p(\Rd;\mathbb C^m)\}.
\end{align*}
On $D_p$ we consider the norm
$\|\uu\|_{D_p}=\|\A_0\uu\|_p+\|v\uu\|_p$,
which is clearly equivalent to the norm $\uu\mapsto \|\uu\|_p+\|\A_0\uu\|_p+\|V\uu\|_p$ due to Hypothesis \ref{hyp_0}(iv) which implies that $\|v\uu\|_p\le \|V\uu\|_p\le c_1\|v\uu\|_p$ and $\|\uu\|_p \le c_0^{-1}\|v\uu\|_p$ for every $\uu\in L^p(\Rd;\C^m)$.
Since $\A_0$ is a closed operator, $D_p$ endowed with the norm $\|\cdot\|_{D_p}$ is a Banach space.

We can now state the main generation result:

\begin{theorem}
\label{mainth}
Under Hypotheses $\ref{hyp_0}$, assume further the condition
\begin{align}
1-\frac{\theta}{p}-\frac{p-1}{p}\kappa m\gamma-\frac{1}{4(p-1)}\bigg [(3-p)\gamma+\frac{\kappa(2-p)}{p}(\sqrt{m}+m)\bigg ]^2>0
\label{pdis}
\end{align}
if $p\in (1,2)$, and the condition
\begin{align}
1\!-\!\frac{\theta}{p}\!-\!\gamma\bigg (\frac{p-1}{p}\kappa m\!+\!\frac{\gamma}{4}(p-1)\!+\!\frac{p-2}{2p}\kappa(\sqrt{m}\!+\!m)\bigg )\!-\!\frac{p-2}{4p^2}\kappa^2m(m\!+\!p\!-\!2)>0,
\label{pdis-1}
\end{align}
if $p\ge 2$. Then, the operator ${\bm A}_p$  generates  an analytic contraction semigroup $\T_p(t)$ in $L^p(\Rd;{\mathbb C}^m)$.
Moreover, if the above assumptions are satisfied also for some $1<q\neq p$, then $\T_p(t)\f=\T_q(t)\f$ for all $\f\in L^p(\Rd; {\mathbb C}^m)\cap L^q(\Rd;{\mathbb C}^m)$.
\end{theorem}

\begin{remark}
\label{forallp}
{\rm We point out that, if Hypothesis \ref{hyp_0}(v) is satisfied for every $\gamma>0$, then  conditions \eqref{pdis} and \eqref{pdis-1} reduce, respectively, to
\begin{align}
4(p^2-\theta p)(p-1)-\kappa^2 (2-p)^2(\sqrt{m}+m)^2>0,
\label{pdis-gamma0}
\end{align}
\begin{align}
4p^2-4\theta p-(p-2)\kappa^2m(m+p-2)>0.
\label{pdis-gamma0-1}
\end{align}
It is easy to check that \eqref{pdis-gamma0-1} is satisfied for every $p\ge 2$ for instance if $\theta<1/2$ and $km<\sqrt{8}$. On the other hand, condition \eqref{pdis-gamma0} cannot be satisfied for every $p\in (1,2)$; it is satisfied, for $p\in [p_0,2)$ for some $p_0\in (1,2)$, for instance if
$4(p_0^2-\theta p_0)(p_0-1)-\kappa^2 (2-p_0)^2(\sqrt{m}+m)^2>0$ and
\begin{eqnarray*}
\Delta=16(\theta^2-\theta+1)+\kappa^4(m+\sqrt{m})^4-32\theta\kappa^2(m+\sqrt{m})^2\le 0.
\end{eqnarray*}
If $\Delta>0$, then the third-order polynomial $f$, in the variable $p$, defined in \eqref{pdis-gamma0} has
a local maximum $f(p_1)$ and a local minimum $f(p_2)$ at some points $0<p_1<p_2$. Clearly, condition $f(p_0)>0$ is necessary
to guarantee that $f(p)>0$ for every $p\in [p_0,2)$. This condition is also sufficient if $p_0>p_2$ or $p_1>2$. On the other hand, when $p_0<p_1<2<p_2$, we also need to require that $\theta<2$ and, when $p_0<p_1<p_2<2$ or $p_1<p_0<p_2\le 2$, we need the additional condition $f(p_2)>0$.

Finally, if inequality \eqref{drift-cont} is replaced by the new condition
\begin{equation}
\sum_{h,k=1}^m \left\vert \sum_{i=1}^d B^i_{hk}(x) \eta^{k}_i \right\vert \leq (\kappa \sqrt{v(x)}+C_{\kappa})\sum_{k=1}^m\langle Q(x)\eta^k,\eta^k\rangle^{\frac{1}{2}}
\label{new-k}
\end{equation}
for every $x,\eta^k\in \Rd$ ($k=1, \dots, m$), $\xi \in \Rm$ and some positive constants $\kappa$ and $C_{\kappa}$, then the generation result in Theorem \ref{mainth} can be applied to the operator $\A-\lambda_{\kappa}$ for a suitable $\lambda_{\kappa}>0$. In particular, the operator $\bm A_p$
generates a strongly continuous analytic semigroup (not contractive, in general)  in $L^p(\Rd;\Rm)$. Indeed, condition \eqref{new-k} implies that for every $\varepsilon>0$ there exists a positive constant $\lambda$ such that Hypotheses \ref{hyp_0} are satisfied with $V$ being replaced by $V+\lambda$, provided that $\varepsilon$ is chosen sufficiently small such that condition \eqref{pdis} (resp. \eqref{pdis-1}) holds true with $\kappa$ being replaced by $\kappa+\varepsilon$.
In particular, if \eqref{new-k} is satisfied by every $\kappa>0$, then conditions
\eqref{pdis} and \eqref{pdis-1} reduce to $\theta<p$. Hence, if $\theta<1$, then we get generation results of a family of consistent semigroups for every $p\in (1,\infty)$.}
\end{remark}

In the proof of Theorem \ref{mainth} we will take advantage of the following result.

\begin{lemma}\label{lemB_variante} For every $\uu\in C_c^\infty(\Rd, \C^m)$, $\eta\in C^1(\Rd)$ and $\varepsilon>0$ it holds that
\begin{align}\label{form_bi}
{\rm Re}\int_{\Rd} \sum_{i=1}^d \langle B^iD_i\uu, \uu\rangle w_{\varepsilon}^{p-2} \eta dx=& \frac{2-p}{2}\sum_{i=1}^d\int_{\Rd}\langle B^i \uu,\uu\rangle w_{\varepsilon}^{p-3} \eta D_iw_{\varepsilon}dx\notag \\
&-\frac{1}{2}\int_{\Rd}\langle ({\rm div}B)\uu,\uu\rangle w_{\varepsilon}^{p-2}\eta dx\notag \\
&- \frac{1}{2}\int_{\Rd} \sum_{i=1}^d\langle B^i\uu,\uu\rangle w_{\varepsilon}^{p-2}D_i\eta dx.
\end{align}
Moreover,
\begin{align}
\int_{\Rd}w_{\varepsilon}^{p-2}\q(w_{\varepsilon})dx
\le \sum_{h=1}^m\int_{\Rd}w_\varepsilon^{p-2}\q(u_h)dx.
\label{form-Q}
\end{align}
\end{lemma}

\begin{proof}
Let us first prove \eqref{form_bi}. We fix $\uu$, $\eta$ and $\varepsilon$ as in the statement of the lemma and observe that for $p=2$, formula \eqref{form_bi} can be obtained just integrating by parts and using the symmetry of the matrices $B^i$ ($i=1,\ldots,d$).\\
In the case $p \neq 2$, we denote by ${\mathscr I}$ the left-hand side of \eqref{form_bi} and set
\begin{eqnarray*}
{\mathscr K}=\frac{1}{p-2}\int_{\Rd} \sum_{i=1}^d \langle B^i \uu, \uu \rangle D_i(w_{\varepsilon}^{p-2})\eta dx=\int_{\Rd} \sum_{i=1}^d \langle B^i \uu, \uu \rangle w_{\varepsilon}^{p-3} D_iw_{\varepsilon}\eta dx.
\end{eqnarray*}
By integrating by parts and taking into account the symmetry of the matrices $B^i$, we deduce that
\begin{align*}
{\mathscr K}=&-\frac{1}{p-2}\int_{\Rd} \langle ({\rm div}B) \uu,\uu\rangle w_{\varepsilon}^{p-2}\eta dx - \frac{2}{p-2} {\rm Re}{\mathscr I}\\
&-  \frac{1}{p-2}\int_{\Rd}\sum_{i=1}^d \langle B^i \uu,\uu\rangle w_{\varepsilon}^{p-2}D_i\eta dx,
\end{align*}
which immediately yields \eqref{form_bi}.

To prove \eqref{form-Q}, we preliminary observe that, since $Q$ is symmetric, it follows that $\langle Q \xi, \xi \rangle= \langle Q {\rm Re}\,\xi,{\rm Re}\,\xi \rangle+ \langle Q {\rm Im}\,\xi, \rm{Im}\,\xi \rangle$ whence $0\le \langle Q {\rm Re}\,\xi,{\rm Re}\,\xi \rangle \le \langle Q \xi, \xi \rangle$ for every $\xi \in \C^d$.
Using this fact and Cauchy-Schwarz inequality, we obtain
\begin{align}\label{num}
\q(w_\varepsilon)= &(4 w_\varepsilon^2)^{-1}\q(|\uu|^2)=w_{\varepsilon}^{-2}\bigg (\sum_{h=1}^m\langle Q{\rm Re}(u_h\nabla \overline{u}_h),{\rm Re}(u_h\nabla \overline{u}_h)\rangle^{\frac{1}{2}}\bigg )^2\notag\\
\le & w_\varepsilon^{-2}\bigg (\sum_{h=1}^m |u_h|\q(u_h)^{\frac{1}{2}}\bigg )^2\le |\uu|^2w_\varepsilon^{-2}\sum_{h=1}^m\q(u_h).
\end{align}
Observing that $|\uu|^2\le w_{\varepsilon}^2$, estimate \eqref{form-Q} follows at once from \eqref{num}.
\end{proof}

\begin{proof}[Proof of Theorem \ref{mainth}]
The proof is articulated in several steps.

{\bf Step 1.} Here, we prove that $(\A, C_c^\infty(\Rd, \C^m))$ is regularly $L^p$-dissipative, namely that there exists $\phi \in (0,\pi/2)$ such that $e^{\pm i \phi}\A$ is dissipative or, equivalently, that there exists a positive constant $C$ such that for all $\uu\in C_c^\infty(\Rd, \C^m)$
\begin{equation}\label{rdiss}
\left|{\rm Im} \int_{\Rd} \langle{\A}\uu,\uu \rangle |\uu|^{p-2} dx\right| \leq  -C{\rm Re} \int_{\Rd} \langle{\A}\uu,\uu \rangle |\uu|^{p-2} dx.
\end{equation}

We fix $\uu\in C_c^\infty(\Rd, \C^m)$ and $\varepsilon>0$.
Recalling that $D_iw_{\varepsilon}=w_{\varepsilon}^{-1}{\rm Re}\langle D_i\uu,\uu\rangle$, an integration by parts shows that
\begin{align}
{\rm Re} \int_{\Rd} \langle {\rm div}(Q\nabla \uu),\uu\rangle w_{\varepsilon}^{p-2}dx
=- \sum_{j=1}^m \int_{\Rd}\q(u_j)w_{\varepsilon}^{p-2}  dx - (p-2) \int_{\Rd}\q(w_{\varepsilon})w_{\varepsilon}^{p-2}dx.
\label{appel}
\end{align}
Hence, using formula \eqref{form_bi}, splitting the term ${\rm Re}\sum_{i=1}^d\int_{\Rd}\langle B^i D_i \uu,\uu\rangle w_{\varepsilon}^{p-2}dx$ into the sum
\begin{eqnarray*}
\frac{p-2}{p}{\rm Re}\sum_{i=1}^d\int_{\Rd}\langle B^i D_i \uu,\uu\rangle w_{\varepsilon}^{p-2}dx+\frac{2}{p}{\rm Re}\sum_{i=1}^d\int_{\Rd}\langle B^i D_i \uu,\uu\rangle w_{\varepsilon}^{p-2}dx
\end{eqnarray*}
and integrating by parts the second integral term, we deduce that
\begin{align}
&-{\rm Re} \int_{\Rd} \langle{\A}\uu,\uu \rangle w_{\varepsilon}^{p-2} dx\notag \\
=&\sum_{j=1}^m \int_{\Rd}\q(u_j)w_{\varepsilon}^{p-2}  dx +(p-2) \int_{\Rd}\q(w_{\varepsilon})w_{\varepsilon}^{p-2}dx\notag \\
&-\frac{p-2}{p} {\rm Re}  \sum_{i=1}^d\int_{\Rd}\langle B^i D_i \uu,\uu\rangle w_{\varepsilon}^{p-2}dx+\frac{p-2}{p}\sum_{i=1}^d\int_{\Rd}\langle B^i\uu,\uu\rangle w_{\varepsilon}^{p-3}  D_iw_{\varepsilon}dx\notag \\
 &+ {\rm Re} \int_{\Rd}\langle(p^{-1}({\rm div}B)+ V)\uu, \uu\rangle w_{\varepsilon}^{p-2}dx.
\label{appel-1}
\end{align}
To ease the notation, we denote by ${\mathscr I}$ and ${\mathscr J}$ the third and fourth integral terms in the right-hand side of the previous formula.
Taking \eqref{drift-cont} into account and applying Cauchy-Schwarz and H\"older's inequalities, for every $\varepsilon_0, \varepsilon_1 >0$ we get
\begin{align*}
{\mathscr I}
\ge & -\frac{|p-2|}{p}\kappa \sum_{k=1}^m\int_{\Rd}\q(u_k)^{\frac 1 2} v^{\frac{1}{2}}|\uu| w_{\varepsilon}^{p-2}dx\\
\ge & - \frac{\kappa \varepsilon_0|p-2|}{p} \sum_{k=1}^m \int_{\Rd}\q(u_k)w_{\varepsilon}^{p-2}dx - \frac{\kappa m|p-2|}{4\varepsilon_0 p}\int_{\Rd} v|\uu|^2w_{\varepsilon}^{p-2}dx
\end{align*}
and
\begin{align*}
{\mathscr J}\ge & -\frac{|p-2|}{p}\kappa m\int_{\Rd}\q(w_{\varepsilon})^{\frac{1}{2}} v^{\frac{1}{2}} |\uu|w_{\varepsilon}^{p-2}dx\notag\\
\ge & -\frac{\kappa m\varepsilon_1|p-2|}{p} \int_{\Rd}\q(w_{\varepsilon})w_{\varepsilon}^{p-2} dx - \frac{\kappa m|p-2|}{4p\varepsilon_1}\int_{\Rd} v |\uu|^2w_{\varepsilon}^{p-2} dx.
\end{align*}
Therefore,
\begin{align}
-{\rm Re} \int_{\Rd} \langle{\A}\uu,\uu \rangle w_{\varepsilon}^{p-2} dx\ge &\left( 1- \frac{\varepsilon_0\kappa |p-2|}{p}\right)\sum_{j=1}^m \int_{\Rd} \q(u_j)w_{\varepsilon}^{p-2} dx\notag \\
&+ \left(p-2-\frac{\varepsilon_1\kappa m }{p}|p-2|\right) \int_{\Rd}\q(w_{\varepsilon})w_{\varepsilon}^{p-2}dx\notag \\
&+ \left[1- \frac \theta p-  \frac{\kappa m|p-2|}{4 p}\left(\frac{1}{\varepsilon_0}+\frac{1}{\varepsilon_1}\right) \right] \int_{\Rd} v |\uu|^2 w_{\varepsilon}^{p-2} dx .
\label{cura}
\end{align}

On the other hand, since
\begin{align*}
|\q({\rm Im}\, u_j, w_{\varepsilon}){\rm Re}\, u_j|+|\q({\rm Re}\, u_j, w_{\varepsilon}){\rm Im}\, u_j|
\le & (\q({\rm Re}\,u_j)^{\frac{1}{2}}+\q({\rm Im}\,u_j)^{\frac{1}{2}})\q(w_{\varepsilon})^{\frac{1}{2}}|\uu|\\
\le &\bigg (\frac{1}{2}\q(u_j)+m\q(w_{\varepsilon})\bigg )w_{\varepsilon}
\end{align*}
for every $j=1,\ldots,m$, we can estimate
\begin{align*}
&\bigg |{\rm Im} \int_{\Rd} \langle{\A}\uu,\uu \rangle w_{\varepsilon}^{p-2} dx\bigg | \\
\le & |p-2|\sum_{j=1}^m\int_{\Rd}(|\q({\rm Im}\, u_j, w_{\varepsilon}){\rm Re}\, u_j|+|\q({\rm Re}\, u_j, w_{\varepsilon}){\rm Im}\, u_j| )w_{\varepsilon}^{p-3}dx\\
&+\sum_{j,k=1}^m\int_{\Rd}\bigg |\sum_{i=1}^d B^i_{jk} D_i u_k\bigg ||\uu| w_{\varepsilon}^{p-2}dx+c_1\int_{\Rd}v |\uu|^2w_{\varepsilon}^{p-2}dx\\
&\leq \frac{|p-2|}{2}\sum_{j=1}^m\int_{\Rd}\q(u_j)w_{\varepsilon}^{p-2}dx + |p-2| m\int_{\Rd}\q(w_{\varepsilon}) w_{\varepsilon}^{p-2}dx\\
&+\kappa \sum_{k=1}^m\int_{\Rd}\q(u_k)^{\frac 1 2}v^{\frac 1 2}|\uu|w_{\varepsilon}^{p-2}dx + c_1\int_{\Rd}v |\uu|^2w_{\varepsilon}^{p-2} dx\\
& \leq \bigg (\frac{|p-2|}{2}+\frac{\kappa}{2}\bigg ) \sum_{j=1}^m\int_{\Rd}\q(u_j)w_{\varepsilon}^{p-2}dx+|p-2|m\int_{\Rd}\q(w_{\varepsilon})w_{\varepsilon}^{p-2}dx\\
&+\bigg (c_1+\frac{\kappa m}{2}\bigg )\int_{\Rd}v|\uu|^2 w_{\varepsilon}^{p-2}dx.
 \end{align*}

We now distinguish between the cases $p\in (1,2)$ and $p\ge 2$. In the first case, the coefficient of the second term in the right-hand side of \eqref{cura} is negative. Using inequality \eqref{form-Q}, we can continue estimate \eqref{cura} and get
\begin{align}
-{\rm Re} \int_{\Rd} \langle{\A}\uu,\uu \rangle w_{\varepsilon}^{p-2} dx\ge &g_1(\varepsilon_0,\varepsilon_1)\sum_{j=1}^m\int_{\Rd}\q(u_j) w_{\varepsilon}^{p-2}dx\notag \\
&+ f_1(\varepsilon_0,\varepsilon_1)\int_{\Rd} v |\uu|^2 w_{\varepsilon}^{p-2} dx,
\label{cura-1}
\end{align}
where
\begin{align}
&f_1(x_1,x_2)=1-\frac{\theta}{p}-\frac{\kappa m|p-2|}{4p}\left (\frac{1}{x_1}+\frac{1}{x_2}\right ),
\label{funct-f1}\\
&g_1(x_1,x_2)=p-1+\frac{p-2}{p}\kappa(x_1+mx_2)
\label{funct-g1}
\end{align}
for every $x_1,x_2>0$, and, similarly,
\begin{align*}
&\bigg |{\rm Im} \int_{\Rd} \langle{\A}\uu,\uu \rangle w_{\varepsilon}^{p-2} dx\bigg |\notag\\
\le &\left(\frac{2-p}{2}+\frac{\kappa}{2}+(2-p)m\right )\sum_{j=1}^m\int_{\Rd}{\mathfrak q}(u_j) w_{\varepsilon}^{p-2}dx
\!+\!\left(c_1+\frac{\kappa m}{2}\right)\int_{\Rd}v|\uu|^2 w_{\varepsilon}^{p-2}dx.
\end{align*}
The supremum of function $f_1$, subject to the constraint $g_1(x_1,x_2)>0$, is
\begin{eqnarray*}
1-\frac{\theta}{p}-\frac{\kappa^2(p-2)^2 m(1+\sqrt{m})^2}{4p^2(p-1)},
\end{eqnarray*}
see Subsection \ref{app-a.3}, which is positive thanks to condition \eqref{pdis}. Then, we can choose $\varepsilon_0$ and $\varepsilon_1$ positive and such that the coefficients of the two terms in the right-hand side
of \eqref{cura-1} are both positive. Thus, we get
\begin{equation}
\bigg |{\rm Im} \int_{\Rd} \langle{\A}\uu,\uu \rangle w_{\varepsilon}^{p-2} dx\bigg | \leq  -C{\rm Re} \int_{\Rd} \langle{\A}\uu,\uu \rangle w_{\varepsilon}^{p-2} dx,
\label{rdiss-epsilon}
\end{equation}
with
\begin{eqnarray*}
C\ge \max\bigg\{\bigg(\frac{2-p}{2} + \frac{\kappa}{2}+(2-p)m\bigg )\frac{1}{g(\varepsilon_0,\varepsilon_1)},
\bigg (c_1+\frac{\kappa m}{2}\bigg )\frac{1}{f_1(\varepsilon_0,\varepsilon_1)}\bigg\}.
\end{eqnarray*}

Now, we address the case $p\ge 2$. Here, the coefficient of the second term in the right-hand side of \eqref{cura} can be made positive
by choosing $\varepsilon_1$ small enough. Note that the supremum of the function $f_1$,
subject to the constraints $x_1\in (0,p/(\kappa(p-2)))$  and $x_2\in (0,p/(\kappa m))$, is
\begin{eqnarray*}
1-\frac{\theta}{p}-\frac{p-2}{4p^2}\kappa^2m(m+p-2),
\end{eqnarray*}
(see Subsection \ref{app-a.4}), which is positive due to condition \eqref{pdis-1}. Hence, we can determine $\varepsilon_0$ and $\varepsilon_1$ such that the coefficients of the three terms in the right-hand side of
\eqref{cura} are all positive. With this choice of the parameters, estimate \eqref{rdiss-epsilon} follows immediately with
\begin{align*}
C\ge\max\bigg\{\frac{p(p-2+\kappa)}{2(p-\varepsilon_0\kappa (p-2))},  \frac{mp}{p-\varepsilon_1\kappa m}, \bigg (c_1+\frac{\kappa m}{2}\bigg )\frac{1}{f_2(\varepsilon_0,\varepsilon_1)}\bigg\}.
\end{align*}
Finally, letting $\varepsilon$ tend to $0^+$ in \eqref{rdiss-epsilon}, by dominated convergence we get \eqref{rdiss} in both cases.

{\bf Step 2.}  Here, we prove that there exists  a constant $C=C(m,p,\gamma, \kappa,  \theta)>0$ such that
\begin{equation}\label{estv}
\|v\uu\|_p\leq C \|\A\uu\|_p,\qquad\;\,\uu\in C^{\infty}_c(\Rd;\C^m),
\end{equation}
assuming that $C_\gamma=0$ in Hypothesis \ref{hyp_0}(v). Clearly, since the coefficients of the operator $\A$ are real-valued, we can limit ourselves to considering functions with values in $\Rm$.
We fix $\uu\in C_c^\infty(\Rd;\Rm)$, $\varepsilon>0$ and set $\f=-\A \uu$. Then,
\begin{align*}
\int_{\Rd} \langle \f, \uu\rangle w_{\varepsilon}^{p-2} v^{p-1} dx=&- \int_{\Rd} \langle {\rm div} (Q\nabla\uu),  \uu\rangle w_{\varepsilon}^{p-2}v^{p-1}dx\notag\\
&-\int_{\Rd} \sum_{i=1}^d \langle B^iD_i\uu, \uu\rangle w_{\varepsilon}^{p-2} v^{p-1}dx\\
&+ \int_{\Rd}\langle V\uu, \uu\rangle w_{\varepsilon}^{p-2}v^{p-1}dx={\mathscr J}_1+{\mathscr J}_2+{\mathscr J}_3.
   \end{align*}
Integrating by parts, taking \eqref{appel} into account, we deduce that
\begin{align*}
{\mathscr J}_1
=&\sum_{j=1}^m \int_{\Rd}\q(u_j)w_{\varepsilon}^{p-2} v^{p-1} dx + (p-2) \int_{\Rd}\q(w_{\varepsilon})v^{p-1}w_{\varepsilon}^{p-2}dx \notag\\
&+\sum_{j=1}^m \int_{\Rd} \q(u_j,v)u_j v^{p-2}w_{\varepsilon}^{p-2}dx + (p-2) \int_{\Rd}\q(w_{\varepsilon},v)v^{p-2} w_{\varepsilon}^{p-1} dx.\notag
\end{align*}
Thus, applying  Cauchy-Schwarz and H\"older's inequalities and taking Hypothesis \ref{hyp_0}(iv) into account, we get
\begin{align*}
{\mathscr J}_1\geq &(1-\varepsilon_1) \sum_{j=1} ^d \int_{\Rd} \q(u_j)w_{\varepsilon}^{p-2} v^{p-1} dx
 - \left( \frac{1}{4\varepsilon_1} + \frac{|p-2|}{4\varepsilon_2}\right) \gamma^2 \int_{\Rd} v^p w_{\varepsilon}^p dx\notag\\
& + (p-2-\varepsilon_2|p-2|)\int_{\Rd}\q(w_{\varepsilon})v^{p-1}w_{\varepsilon}^{p-2}dx
\end{align*}
for every $\varepsilon_1, \varepsilon_2 >0.$

We now estimate the term ${\mathscr J}_2$. Using the same arguments as in the
proof of \eqref{appel-1} and applying Young's inequality, we get
\begin{align*}
{\mathscr J}_2\geq &-\frac{|p-2|}{p} \kappa \sum_{k=1}^m \int_{\Rd}\q(u_k)^{\frac{1}{2}}|\uu| w_{\varepsilon}^{p-2}v^{p-\frac{1}{2}}dx\\
&-\frac{|p-2|}{p} \kappa m\int_{\Rd}\q(w_{\varepsilon})^{\frac{1}{2}}|\uu|w_{\varepsilon}^{p-2}v^{p-\frac{1}{2}}dx
+\frac{1}{p}\int_{\Rd}\langle ({\rm div}B)\uu,\uu\rangle w_{\varepsilon}^{p-2}v^{p-1} dx \\
&-\frac{p-1}{p}\kappa m\int_{\Rd}\q(v)^{\frac{1}{2}}v^{p-\frac{3}{2}}|\uu|w_{\varepsilon}^{p-1} dx\\
\geq & - \kappa\frac{|p-2|}{p} \varepsilon_3\sum_{k=1}^m \int_{\Rd}\q(u_k)w_{\varepsilon}^{p-2}v^{p- 1}dx\\
&- \frac{|p-2|}{p}\kappa m\varepsilon_4\int_{\Rd}\q(w_{\varepsilon})w_{\varepsilon}^{p-2}v^{p-1}dx
+\frac{1}{p}\int_{\Rd} \langle ({\rm div}B)\uu,\uu\rangle w_{\varepsilon}^{p-2}v^{p-1} dx\\
&-\left[ \frac{\kappa|p-2|}{4p}m\left( \frac{1}{\varepsilon_3}+ \frac{1}{\varepsilon_4}\right) + \frac{\gamma \kappa m(p-1)}{p} \right]\int_{\Rd}|\uu|^2w_{\varepsilon}^{p-2}v^pdx
\end{align*}
for every $\varepsilon_3, \varepsilon_4>0$.

Summing up, we have proved that
\begin{align}
& \int_{\Rd} \langle \f, \uu\rangle w_{\varepsilon}^{p-2} v^{p-1} dx\ge \bigg( 1-\varepsilon_1-\kappa \frac{|p-2|}{p}\varepsilon_3\bigg )\sum_{k=1}^m \int_{\Rd}\q(u_k)w_{\varepsilon}^{p-2}v^{p-1}dx \notag\\
&+ \bigg [ p-2-|p-2|\bigg(\varepsilon_2 + \frac{\kappa m}{p}\varepsilon_4\bigg)\bigg ] \int_{\Rd}\q(w_{\varepsilon}) w_{\varepsilon}^{p-2}v^{p-1}dx\notag\\
&+f_2(\varepsilon_1,\varepsilon_2,\varepsilon_3,\varepsilon_4)\int_{\Rd} v^p|\uu|^2w_{\varepsilon}^{p-2}dx,
\label{fine}
\end{align}
where
\begin{align}
f_2(x_1,x_2,x_3,x_4)\!=\!1\!-\!\frac{\theta}{p}\!-\!\gamma^2\bigg (\!\frac{1}{4x_1}\!+\!\frac{|p\!-\!2|}{4x_2}\!\bigg )\!-\!\frac{\kappa m|p\!-\!2|}{4p}\bigg (\!\frac{1}{x_3}\!+\!\frac{1}{x_4}\!\bigg )\!-\!\frac{p\!-\!1}{p}\kappa m \gamma
\label{funct-f2}
\end{align}
for every $(x_1,x_2,x_3,x_4)\in\R^4_+$.
By applying Young's inequality, we can easily show that for every $\delta>0$ there exists a positive constant $C=C(\delta, p)$ such that
\begin{equation}
\int_{\Rd} \langle \f, \uu\rangle w_{\varepsilon}^{p-2} v^{p-1} dx\leq \delta \int_{B(0,R)} v^p w_{\varepsilon}^p dx + C\int_{\Rd} |\f|^pdx,
\label{q3}
\end{equation}
where $B(0,R)$ is any ball containing the support of the function $\f$.
By combining \eqref{fine} and \eqref{q3}, we get
\begin{align}
&\left( 1-\varepsilon_1-\kappa \frac{|p-2|}{p}\varepsilon_3\right)\sum_{k=1}^m\int_{\Rd}\q(u_k)w_{\varepsilon}^{p-2}v^{p-1}dx\notag \\
&+ \bigg [p-2-|p-2|\bigg (\varepsilon_2+\frac{\kappa m}{p}\varepsilon_4\bigg )\bigg ]\int_{\Rd}\q(w_{\varepsilon}) w_{\varepsilon}^{p-2}v^{p-1}dx\notag\\
&+f_2(\varepsilon_1,\varepsilon_2,\varepsilon_3,\varepsilon_4)\int_{\Rd}|\uu|^2w_{\varepsilon}^{p-2}v^pdx-\delta \|v w_{\varepsilon}\|_{L^p(B(0,R))}^p\leq C\int_{\Rd} |\f|^pdx.
\label{domande}
\end{align}
As in Step 1, we distinguish between the cases $p\in (1,2)$ and $p\ge 2$. In the first case, assuming that $\displaystyle 1-\varepsilon_1-p^{-1}\kappa |p-2|\varepsilon_3>0$ and using \eqref{form-Q}, we
can combine the first two terms in the left-hand side of \eqref{domande} and obtain the inequality
\begin{align}
&g_2(\varepsilon_1,\varepsilon_2,\varepsilon_3,\varepsilon_4)\int_{\Rd}\q(w_{\varepsilon})w_{\varepsilon}^{p-2}v^{p-1}dx\notag\\
&+f_2(\varepsilon_1,\varepsilon_2,\varepsilon_3,\varepsilon_4)\int_{\Rd}|\uu|^2w_{\varepsilon}^{p-2}v^pdx
-\delta\|v w_{\varepsilon}\|_{L^p(B(0,R))}^p
\leq C\int_{\Rd} |\f|^pdx,
\label{domande-1}
\end{align}
where the function $g_2:\R_+^4\to\R$ is defined by
\begin{align}
g_2(x_1,x_2,x_3,x_4)=&p-1-x_1-\kappa \frac{|p-2|}{p}x_3-(2-p)\bigg (x_2+\frac{\kappa m}{p}x_4\bigg )
\label{funct-g2}
\end{align}
for every $(x_1,\ldots,x_4)\in\R^4_+$.
It is easy to check that the supremum of $f_2$, subject to the constrain $g_2(x_1,\ldots,x_4)>0$, is given by
\begin{align*}
1-\frac{\theta}{p}-\frac{p-1}{p}\kappa m\gamma-\frac{1}{4(p-1)}\bigg [(3-p)\gamma+\frac{\kappa|p-2|}{p}(\sqrt{m}+m)\bigg ]^2,
\end{align*}
(see Subsection \ref{app-a.1} for further details).
Due to condition \eqref{pdis}, this supremum is positive. Thus, we can choose the parameters $\varepsilon_j$ ($j=1,\ldots,4$)
such that the coefficients of the two first integral terms in the left-hand side of \eqref{domande-1} are both positive, so that
\begin{align*}
f_2(\varepsilon_1,\varepsilon_2,\varepsilon_3,\varepsilon_4)\int_{\Rd}|\uu|^2w_{\varepsilon}^{p-2}v^pdx
-\delta\|v w_{\varepsilon}\|_{L^p(B(0,R))}^p\leq &C\int_{\Rd} |\f|^pdx.
\end{align*}
Letting $\varepsilon$ tend to $0^+$, we can choose $\delta>0$ such that estimate \eqref{estv} follows.

If $p\ge 2$, then the supremum of the function $f_2$, subject to the constrains
$p(1-x_1)-\kappa (p-2)x_3>0$ and $p(1-x_2)-\kappa m x_4>0$, is
\begin{eqnarray*}
1-\frac{\theta}{p}-\frac{p-1}{p}\kappa m\gamma-\frac{\gamma^2}{4}(p-1)-\frac{p-2}{2p}\kappa\gamma(\sqrt{m}+m)-\frac{p-2}{4p^2}\kappa^2m(m+p-2),
\end{eqnarray*}
(see Subsection \ref{app-a.2} for further details), which is positive due to condition \eqref{pdis-1}. Now, we can argue as in the case $p\in (1,2)$ to obtain estimate \eqref{estv}.

{\bf Step 3.}  Here, we prove that there exist positive constants $M_1$ and $M_2$ depending on $\kappa , c_0,c_1,  c_2, \gamma, C_\gamma$, $\theta$ and $p$,  such that
\begin{equation}
M_1 \|\uu\|_{D_p}\leq \|\A\uu- \uu \|_p \leq M_2 \|\uu\|_{D_p},\qquad\;\,\uu\in C_c^\infty(\Rd;\C^m).
 \label{3-3}
 \end{equation}
Also in this case, we can assume that $\uu$ takes values in $\Rm$. We first assume that $C_\gamma=0$ and observe that condition \eqref{pdis} implies that $\gamma^2>4(p-1)^{-1}$ if $p\in (1,2)$. Then,  for every $\uu\in C_c^\infty(\Rd;\Rm)$, taking into account Step $2$ and \eqref{interpolazione}, we can estimate
 \begin{align*}
\|\uu\|_{D_p}
 \le & \|\A\uu\|_p  + \bigg\|\sum_{i=1}^d B^i D_i\uu\bigg\| _p+\|V\uu\|_p + \|v\uu\|_p\\
\le &  \|\A\uu\|_p + \frac{1}{2} \|\A_0\uu\|_p +(1+K_{1/2})c_1 \|v\uu\|_p \\
 \le & \frac{1}{2}\|\uu\|_{D_p}+[1+(1+K_{1/2})c_1C]\|\A\uu\|_p,
 \end{align*}
so that using Step 1, which shows that the operator $\A$ is dissipative, we get
$\|\uu\|_{D_p}\le 2K\|\A\uu\|_p\le 2K\|\A\uu-\uu\|_p+2K\|\uu\|_p,
\le 4K\|\A\uu-\uu\|_p$,
where $K=1+(1+K_{1/2})c_1C$. The first part of \eqref{3-3} follows with $M_1=(4K)^{-1}$.
If $C_\gamma\neq 0$, then we can determine a positive constant $\lambda$ such that $V+\lambda I$ and $v+\lambda$ satisfy Hypothesis \ref{hyp_0}(ii) with $C_\gamma=0$.
 In such a case, using again the dissipativity of $\A$ we obtain
\begin{align*}
\|\uu\|_{D_p}\le\widetilde K\|\A\uu-\lambda\uu-\uu\|_p\le \widetilde K\|\A\uu-\uu\|_p+\widetilde K\lambda\|\uu\|_p\le (1+\lambda)\widetilde K\|\A\uu-\uu\|_p
\end{align*}
and the first part of \eqref{3-3} follows with $M_1=((1+\lambda)\widetilde K)^{-1}$ and $\widetilde K$ is a positive constant, depending on $c_1$, $K_{1/2}$ and $C$.

To prove the other part of \eqref{3-3} we argue similarly, observing that
\begin{align*}
\|\A\uu -\uu\|_p \leq &2\|\A_0\uu\|_p+(M_1+1)\|V\uu\|_p+\|\uu\|_p\\
 \le &2\|\A_0\uu\|_p+(c_1M_1+c_1+c_0^{-1})\|v\uu\|_p\\
\le &\max\{2,c_1M_1+c_1+c_0^{-1}\}\|\uu\|_{D_p}.
\end{align*}

{\bf Step 4.} Here, we prove that $D_p=D_{p,\max}(\A_0-V):=\{\uu\in L^p(\Rd;\C^m): \A_0\uu-V\uu\in L^p(\Rd;\C^m)\}$ and that $C^{\infty}_c(\Rd;\Rm)$ is dense in $D_p$. Clearly, $D_p\subset D_{p,\max}(\A_0-V)$, so let us prove the other inclusion. We fix $\uu\in D_{p,\max}(\A_0-V)$ and observe that Theorem \ref{thm-1} guarantees
the existence of a sequence $(\uu_n)\subset C^{\infty}_c(\Rd;\R^m)$ converging to $\uu$ in $L^p(\Rd;\R^m)$ and such that $\A_0\uu_n-V\uu_n$
converges to $\A_0\uu-V\uu$ in $L^p(\Rd;\R^m)$ as $n$ tends to $\infty$.
Step 3, applied with $B^i=0$, shows that $(\uu_n)$ is a Cauchy sequence in $D_p$ endowed with the norm $\|\cdot\|_{D_p}$.
Since this latter is a Banach space, we conclude that $\uu\in D_p$ and the inclusion $D_{p,\max}(\A_0-V)\subset D_p$ follows.
This argument also shows that $C^{\infty}_c(\Rd;\R^m)$ is dense in $D_p$.

{\bf Step 5.} Here, we prove that the operator $(\A, D_p)$ generates a contraction semigroup in $L^p(\Rd;\Rm)$.
In view of Step 1, it suffices to show that the operator $I-\A: D_p\to L^p(\Rd;\Rm)$ is surjective.
For this purpose, we apply the continuity method. For every $t\in [0,1]$, we introduce the operator
\begin{align*}
\bm{\mathcal L}_t  \uu= \uu -  {\rm div}(Q \nabla \uu)- t\sum_{i=1}^d B^i D_i \uu + V\uu, \qquad\;\,  \uu\in D_p.
\end{align*}
Due to the density of $C^{\infty}_c(\Rd;\Rm)$ in $D_p$, we can first extend \eqref{interpolazione} to every $\uu\in D_p$ and, then, using this inequality, we can extend \eqref{3-3} to every $\uu\in D_p$. Thus, we can determine a positive constant $K$, independent of $t\in [0,1]$, such that
$\|\bm{\mathcal L}_t\uu\|_{p} \geq K \|\uu\|_{D_p}$ for every $t\in [0,1]$ and $\uu\in D_p$. Moreover, Theorem \ref{thm-1} and Step 4 show that the operator $(\A_0-V,D_p)$ generates
a strongly continuous semigroup of contractions in $L^p(\Rd;\Rm)$. Hence, the operator $\bm {\mathcal L}_0$ is surjective on $L^p(\Rd;\Rm)$. Since the operator $\A_0$ has real-valued coefficients, it follows that $\bm{\mathcal L}_0$ is surjective on $L^p(\Rd;\Rm)$. The continuity method applies
showing that also the operator $\bm{\mathcal L}_1=I-\A:D_p\rightarrow L^p(\Rd;\Rm)$ is surjective.

{\bf Step 6.} Here, we complete the proof, showing that, if $q$ is a different index in $(1,\infty)$, which satisfies the assumptions of the theorem, then
$\T_p(t)\f=\T_q(t)\f$ for every $t>0$ and $\f\in L^p(\Rd;\mathbb C^m)\cap L^q(\Rd;\mathbb C^m)$.
Since both $\T_p(t)$ and $\T_q(t)$ map functions with values in $\R^m$ in functions with values in $\Rm$, we can limit ourselves to considering functions with values in $\R^m$.
By Theorem \ref{thm-1}, the semigroups generated by the closure of the operator $(\A_0-V,C^{\infty}_c(\Rd;\R^m))$ in $L^p(\Rd;\R^m)$ and in $L^q(\Rd;\R^m)$, coincide on $L^p(\Rd;\R^m)\cap L^q(\Rd;\R^m)$. As a byproduct, writing the
resolvent operators as the Laplace transform of the semigroups. we infer that the resolvent operators coincide on $L^p(\Rd;\R^m)\cap L^q(\Rd;\R^m)$. Therefore,
for every $\f\in L^p(\Rd;\R^m)\cap L^q(\Rd;\R^m)$ and $\lambda\in\R$ sufficiently large there exists a unique $\uu\in D_p\cap D_q$ which solves the equation $\lambda\uu-\A_0\uu+V\uu=\f$.

Next, we observe that all the computations in the previous steps can be performed replacing $\|\cdot\|_{D_p}$ with $\|\cdot\|_{D_p} + \|\cdot\|_{D_q}$ and by applying the method of continuity in the space $L^p(\Rd; \Rm)\cap L^q(\Rd;\Rm)$ endowed with the norm $\|\cdot\|_p+\|\cdot\|_q$.
It follows that $\lambda I-\A: D_p\cap D_q\rightarrow L^p(\Rd; \Rm)\cap L^q(\Rd;\Rm)$ is invertible for every $\lambda >0$, and therefore $(\lambda I-
{\bm A}_p)^{-1}\f=(\lambda I-{\bm A}_q)^{-1}\f$ for every $\f\in L^p(\Rd; \Rm)\cap L^q(\Rd;\Rm)$.
By the representation formula of semigroups in terms of the resolvents, we get the assertion.
\end{proof}

\begin{remark}
\rm{We point out that if there exists $\mu_0>0$ such that $\langle Q(x)\xi, \xi  \rangle \ge \mu_0|\xi|^2$  for every $x,\xi \in \Rd$, then $D_2$ is continuously embedded in $W^{1,2}(\Rd;\Rm)$. Indeed, from \eqref{domande}, with $v\equiv 1$, it follows that
\begin{eqnarray*}
\sum_{k=1}^m\int_{\Rd}\q(u_k)dx\le C'\int_{\Rd}|\A\uu|^2dx,\qquad\;\,\uu\in C^{\infty}_c(\Rd;\C^m),
\end{eqnarray*}
from which it follows immediately that
\begin{eqnarray*}
\int_{\Rd}|\nabla\uu|^2dx\le C'\mu_0^{-1}\int_{\Rd}|\A\uu|^2dx.
\end{eqnarray*}
Since $C^{\infty}_c(\Rd;\C^m)$ is a core for ${\bm A}_2$, the previous inequality extends to every $\uu\in D_2$.}
\end{remark}

To conclude this section we prove that $D_p$ coincides with the maximal domain of the realization ${\bm A}_p$ of $\A$ in $L^p(\Rd;\Rm)$ and we provide some examples of operators $\A$ which satisfy our assumptions.

\begin{proposition}
Under the assumptions of Theorems $\ref{core}$ and $\ref{mainth}$, it holds that
\begin{eqnarray*}
D_p=\{\uu \in L^p(\Rd;{\mathbb C}^m)\cap W^{2,p}_{\rm loc}(\Rd;\C^m): \A\uu \in L^p(\Rd;\C^m)\}=: D_{\rm max}({\bm A}_p)
\end{eqnarray*}
for each $p \in (1,\infty)$. Consequently, $C^\infty_c(\Rd;\C^m)$ is a core for $({\bm A}_p,D_{\rm max}({\bm A}_p))$.
\end{proposition}

\begin{proof}
First of all, let us observe that the inclusion $D_p \subset D_{\rm max}(\A_p)$ is immediate consequence of the estimate $\|\uu\|_p\le c_0^{-1}\|v\uu\|_p$ and the interpolative estimate \eqref{interpolazione}.
To prove that $D_{\rm max}(\A_p)\subset D_p$, it suffices to prove that $\lambda I-\A$ is injective on $D_{\rm max}(\A_p)$ for some (hence all) $\lambda>0$. So, let us consider  $\uu \in D_{\rm max}(\A_p)$ such that $\lambda \uu-\A \uu={\bm 0}$. We have to show that $\uu \equiv {\bm 0}$. To this aim, we prove that
\begin{equation}\label{caos}
\lambda\int_{\Rd} |\uu|^2z^{p-2}  dx\le 0,
\end{equation}
where $z=w_{\varepsilon}$ and $\varepsilon$ is any positive constant, if $p\in (1,2)$, whereas $z=|\uu|$ if $p\ge 2$.
Once formula \eqref{caos} is proved, the claim follows easily letting $\varepsilon \to 0$. The argument used to prove \eqref{caos} is similar to that already used in the proof of Theorems \ref{core} and \ref{mainth}. For this reason we give a sketch of it.
Note that
\begin{align}
\lambda\int_{\Rd}|\uu|^2 z^{p-2}\zeta_n^2 dx&=- \sum_{k=1}^m \int_{\Rd}{\mathfrak q}(u_k)z^{p-2}\zeta_n^2 dx- (p-2)\int_{\Rd} {\mathfrak q}(z)z^{p-2}\zeta_n^2dx\notag\\
&-  \int_{\Rd} {\mathfrak q}(z,\zeta_n^2)z^{p-1} dx-\frac{1}{p}\sum_{i=1}^d\int_{\Rd}\langle B^i \uu,\uu\rangle z^{p-2}D_i\zeta_n^2 dx\notag\\
&+ \frac{p-2}{p}\int_{\Rd} \sum_{i=1}^d \langle B^iD_i\uu, \uu\rangle z^{p-2} \zeta_n^2dx\notag\\
&+\frac{2-p}{p}\sum_{i=1}^d\int_{\Rd}\langle B^i \uu,\uu\rangle z^{p-3}\zeta_n^2D_i z dx\notag\\
&-\int_{\Rd}[p^{-1}\langle ({\rm div}B)\uu,\uu\rangle+\langle V\uu,\uu\rangle]  z^{p-2} \zeta_n^2 dx.
\label{para}
\end{align}
To ease the notation, we denote by ${\mathscr I}_j$, $j=1,\ldots,5$ the last five integral terms in the right-hand side of \eqref{para}.
Using H\"older's inequality, we get
\begin{align*}
|{\mathscr I}_1|\le \varepsilon_0 \int_{\Rd}{\mathfrak q}(z) z^{p-2}\zeta_n^2dx+ \frac{1}{\varepsilon_0}\int_{\Rd}{\mathfrak q}(\zeta_n)z^{p} dx
\end{align*}
for every $\varepsilon_0>0$.
Moreover, using Hypothesis \ref{hyp_0}(iv) and again H\"older's inequality  we deduce that
\begin{align*}
|{\mathscr I}_2|&\le 2\frac{m\kappa}{p}  \int_{\Rd}\sqrt{v}\sqrt{{\mathfrak q}(\zeta_n)}|\uu|^2 z^{p-2}\zeta_n dx\\
& \le \varepsilon_2 \int_{\Rd} v|\uu|^2z^{p-2}\zeta_n^2 dx+\frac{m^2\kappa^2}{\varepsilon_2p^2}\int_{\Rd}{\mathfrak q}(\zeta_n)|\uu|^2 z^{p-2}dx
\end{align*}
for every $\varepsilon_2>0$ and
\begin{align*}
|{\mathscr I}_3|
\le \varepsilon_1 \kappa\frac{|p-2|}{p} \sum_{k=1}^m\int_{\Rd}{\mathfrak q}(u_k)z^{p-2}\zeta_n^2dx+ \frac{m\kappa}{4\varepsilon_1}\frac{|p-2|}{p}\int_{\Rd} v|\uu|^2z^{p-2} \zeta_n^2dx
\end{align*}
for every $\varepsilon_1>0$.
Finally, for every $\varepsilon_3>0$,
\begin{align*}
|{\mathscr I}_4| &\le \varepsilon_3m \kappa\frac{|p-2|}{p} \int_{\Rd}{\mathfrak q}(z)\zeta_n^2 z^{p-2}dx+ \frac{m\kappa}{4 \varepsilon_3}\frac{|p-2|}{p}\int_{\Rd}|\uu|^2 v z^{p-2}\zeta_n^2 dx.
\end{align*}
In addition,
\begin{eqnarray*}
-{\mathscr I}_5\le \left(\frac{\vartheta}{p}-1\right)\int_{\Rd} v|\uu|^2 z^{p-2} \zeta_n^2 dx.
\end{eqnarray*}
Now, we distinguish between the cases $p \ge 2$ and $p \in(1,2)$. In the first case, from all the above estimates we obtain
\begin{align}\label{last}
\lambda\int_{\Rd}|\uu|^p\zeta_n^2 dx\le &\left(-1+\varepsilon_1\kappa\frac{p-2}{p}\right) \sum_{k=1}^m \int_{\Rd}{\mathfrak q}(u_k)|\uu|^{p-2} \zeta_n^2 dx\notag\\
&+\left(2-p+\varepsilon_0+\frac{p-2}{p}\kappa m\varepsilon_3\right)\int_{\Rd}{\mathfrak q}(|\uu|)|\uu|^{p-2}\zeta_n^2dx\notag\\
&-(f_1(\varepsilon_1,\varepsilon_3)-\varepsilon_2)\int_{\Rd} v|\uu|^p\zeta_n^2 dx\notag\\
&+ \left(\frac{1}{\varepsilon_0}+\frac{m^2\kappa^2}{p^2\varepsilon_2}\right)\int_{\Rd}{\mathfrak q}(\zeta_n)|\uu|^p dx,
\end{align}
where function $f_1$ is defined by \eqref{funct-f1}.
Now, thanks to condition \eqref{pdis-1}, with $\gamma=0$, we can choose $\varepsilon_i>0$ ($i=0,1,2,3$) to ensure that the first three terms in the right hand side of \eqref{last} are nonpositive.
We refer the reader to Subsection \ref{app-a.3} for further details. Thus, we get
\begin{eqnarray*}
\lambda\int_{\Rd}|\uu|^p\zeta_n^2 dx\le C\int_{\Rd} \q(\zeta_n)|\uu|^{p} dx
\end{eqnarray*}
for some positive constant $C$ depending on $m,p,\kappa$. Then, arguing as in the proof of Theorem \ref{core}, letting $n \to\infty$ we deduce \eqref{caos}.\\
 In the second case, when $p \in (1,2)$ we use estimate \eqref{form-Q} to deduce that
 \begin{align}\label{last-1}
 \lambda\int_{\Rd}|\uu|^2w_{\varepsilon}^{p-2}\zeta_n^2 dx&\le [\varepsilon_0-g_1(\varepsilon_1,\varepsilon_3)]\sum_{k=1}^m \int_{\Rd} \q(u_k)w_{\varepsilon}^{p-2} \zeta_n^2 dx\notag\\
 &+[\varepsilon_2-f_2(\varepsilon_1,\varepsilon_3)]\int_{\Rd} v|\uu|^2 w_{\varepsilon}^{p-2} \zeta_n^2 dx\notag\\
&+ \left(\frac{1}{\varepsilon_0}+\frac{m^2\kappa^2}{p^2\varepsilon_2}\right)\int_{\Rd}\q(\zeta_n)w_{\varepsilon}^{p} dx,
 \end{align}
 where the functions $g_1$ and $f_2$ are defined by \eqref{funct-g1} and \eqref{funct-f2}.
 Also in this case, using condition \eqref{pdis}, with $\gamma=0$, we can choose $\varepsilon_i>0$ ($i=0,1,2,3$) to make the first two terms in the right hand side of \eqref{last-1} nonpositive and then we can conclude as in the first case. We refer the reader to Subsection \ref{app-a.4} for further details.
\end{proof}

\begin{example}
\label{example-1}
{\rm Let $\A$ be the operator defined in \eqref{opA-intro} with
\begin{equation}\label{diffusion}
q_{ij}(x)=(1+|x|^2)^{\frac{\alpha}{2}}\delta_{ij},\qquad\;\, i,j=1, \ldots,d,
\end{equation}
\begin{equation}\label{drift}
B^i(x)=(1+|x|^2)^\gamma e^{\frac{1}{2}(1+|x|^2)^\beta}A^i, \qquad\;\, i=1,\ldots,d,
\end{equation}
for every $x\in\Rd$, and let $V:\Rd\to \R^{m\times m}$ be a measurable function such that
\begin{equation}\label{exp-V}
\langle V(x)\xi ,\xi\rangle \ge v(x)|\xi|^2,\qquad\;\, |V(x)\xi|\le c_1v(x)|\xi|
\end{equation}
with $v(x):=e^{(1+|x|^2)^\beta}$ for $x\in\Rd$. Here, $\beta\in (0,\infty)$, $\alpha\in [0,2]$, $\gamma \in [0,\alpha/4]$ and $A^i$ is a symmetric $m\times m$ real-valued matrix for every $i=1,\ldots, d$ and $j,k=1, \ldots, m$.

It is easy to check that, for every $\varepsilon >0$, there exists a positive constant $C_{\varepsilon}$ such that
\begin{equation}\label{etoile}
|Q(x)|^{\frac{1}{2}}|\nabla v(x)|\le \varepsilon v(x)^{\frac{3}{2}}+C_{\varepsilon},\qquad\;\, x\in \Rd.
\end{equation}
Moreover the condition $\alpha \in [0,2]$ ensures that
$\langle Q\nabla v,\nabla v\rangle \le Cv^2\log^2v$ on $\Rd$ for some constant $C>0$. Since $|\nabla Q(y)|\le C_*(1+|y|^2)^{\frac{\alpha-1}{2}}$ for every $y\in\Rd$ and $\rho(x)=C_{**}(1+|x|^2)^{\frac{\alpha}{4}}e^{-\frac{1}{2}(1+|x|^2)^{\beta}}$ for every $x\in\Rd$ and some positive constants $C_*$ and $C_{**}$, it follows that
\begin{eqnarray*}
|\nabla Q(y)|\le C_*(1+(|x|+C_{**}(1+|x|^2)^{\frac{\alpha}{4}}e^{-\frac{1}{2}(1+|x|^2)^{\beta}})^2)^{\frac{\alpha-1}{2}},\qquad\;\,y\in B(x,\rho(x)).
\end{eqnarray*}
Consequently the inequality $\sup_{|x-y|\le\rho(x)}|\nabla Q(y)|\le c_2(1+|x|^2)^{\frac{\alpha}{2}}(\rho(x))^{-1}$ for every $x\in\Rd$ follows immediately, with a positive constant $c_2$, since the right-hand side grows faster than the left-hand side as $|x|$ tends to $\infty$.

On the other hand, if we set
\begin{eqnarray*}
A_0:= \max_{k\in\{1,\ldots,m\}}\sum_{h=1}^m\bigg (\sum_{i=1}^d (A^i_{hk})^2\bigg )^{\frac{1}{2}},
\end{eqnarray*}
then, by applying the Cauchy-Schwarz inequality, we get
\begin{align*}
\sum_{h,k=1}^m\bigg |\sum_{i=1}^d B^i_{hk}\eta_i^k\bigg |=&(1+|x|^2)^\gamma e^{\frac{1}{2}(1+|x|^2)^\beta}\sum_{h,k=1}^m\bigg |\sum_{i=1}^d A^i_{hk}\eta_i^k\bigg |\\
&\le A_0\sqrt{v(x)}(1+|x|^2)^\gamma\sum_{k=1}^m|\eta^k|\\
& \le A_0\sqrt{v(x)}(1+|x|^2)^{\gamma-\frac{\alpha}{4}}\sum_{k=1}^m \langle Q \eta^k,\eta^k\rangle^{\frac{1}{2}}\\
& \le A_0\sqrt{v(x)}\sum_{k=1}^m \langle Q \eta^k,\eta^k\rangle^{\frac{1}{2}}
\end{align*}
for every $x, \eta^k\in \Rd$. Denote by $\lambda_{A^i}$, $\Lambda_{A^i}$, respectively, the minimum and the maximum eigenvalues of $A^i$ and let $\lambda\in\Rd$ the vector with entries $\lambda_i=|\lambda_{A^i}|\vee|\Lambda_{A^i}|$ ($i=1,\ldots,d$). Then, again by the Cauchy-Schwarz inequality, we can estimate
\begin{align*}
\langle ({\rm div}B(x)) \xi,\xi \rangle
& \ge - (1+|x|^2)^{\gamma-1}e^{\frac{1}{2}(1+|x|^2)^\beta}\left(2\gamma +\beta (1+|x|^2)^{\beta}\right)\sum_{i=1}^d|\langle A^i \xi,\xi\rangle|| x_i|\\
& \ge -(1+|x|^2)^{\gamma-1}e^{\frac{1}{2}(1+|x|^2)^\beta}\left(2\gamma +\beta (1+|x|^2)^{\beta}\right)|\xi|^2|\lambda| |x|\\
& \ge -\vartheta v(x)|\xi|^2
\end{align*}
for every $x\in\Rd$ and $\xi\in\Rm$, where $\vartheta:= (2\gamma+\beta)|\lambda| c_0$ and
\begin{align*}
c_0:=\max_{x\in \Rd}\{|x|(1+|x|^2)^{\beta+\gamma-1}e^{-\frac{1}{2}(1+|x|^2)^\beta}\}.
\end{align*}
Therefore, Hypotheses \ref{hyp_0} are satisfied for every $p>c_0(2\gamma+\beta)|\lambda|$.
If we assume further that $4(p^2-\theta p)(p-1)-A_0^2 (2-p)^2(\sqrt{m}+m)^2>0$, then, by Theorem \ref{mainth} and
(see also Remark \ref{forallp}) and \cite[Theorem 2.7]{FL07}, the operator
\begin{align*}
\bm A_p\f =& {\rm div}((1+|x|^2)^{\frac{\alpha}{2}}\nabla \f)+(1+|x|^2)^\gamma e^{\frac{1}{2}(1+|x|^2)^\beta}\sum_{i=1}^dA^iD_i\f-V\f,
\end{align*}
with domain
\begin{align*}
D(\bm{A}_p)= \{\f\in W^{2,p}(\Rd ,\C^m): (1+|x|^2)^{\frac{\alpha}{2}}|D^2\f|,\,(1+|x|^2)^{\frac{\alpha}{4}}|\nabla \f|,\ |V\f|\in L^p(\Rd)\},
\end{align*}
generates an analytic contraction semigroup $\T_p(t)$ in $L^p(\Rd ,\C^m)$. Moreover, $\T_p(t)\f$$=\T_q(t)\f$ for all $\f\in L^p(\Rd ,\C^m)\cap L^q(\Rd ,\C^m)$ and  $q>c_0(2\gamma+\beta)|\lambda|$ such that $4(q^2-\theta q)(q-1)-A_0^2 (2-q)^2(\sqrt{m}+m)^2>0$.

An example of function $V$ satisfying \eqref{exp-V} for $m=2$ is
\begin{equation*}
V(x) =
\begin{pmatrix}
 e^{(1+|x|^2)^\beta} &\phi(x) \\
-\phi(x) & e^{(1+|x|^2)^\beta}
\end{pmatrix}
\end{equation*}
for every measurable function $\phi$ satisfying $|\phi(x)|\le e^{(1+|x|^2)^\beta}$ for all $x\in \Rd$.}
\end{example}

\section{The semigroup in $C_b(\Rd;\R^m)$}

In this section, under suitable assumptions on the coefficients of the operator $\A$, we prove that we can associate a semigroup ${\bm T}_{\infty}(t)$ with $\A$ in $C_b(\Rd;\C^m)$ that, under the assumptions of Theorem \ref{mainth}, coincides with the semigroup $\T(t)$ generated in $L^p(\Rd;\C^m)$ for each $p\in (1,\infty)$. Again, we can limit ourselves to considering functions with values in $\R^m$.
\begin{hyp}\label{hyp_cb}
\begin{enumerate}[\rm(i)]
\item
The coefficients $q_{ij}$ belong to $C^{1+\alpha}_{\rm loc}(\Rd)$ for $i, j=1, \ldots,d$ and $\langle Q(x)\xi,\xi\rangle \ge\mu_0|\xi|^2$ for every $x,\xi\in\Rd$ and some positive constant $\mu_0$; the functions $B^i$ $(i=1, \ldots,d)$ satisfy Hypotheses $\ref{hyp_0}(ii)$, $(iv)$; the function $V$ has entries in $C^{\alpha}_{\rm loc}(\Rd)$ for some $\alpha \in (0,1)$ and $\langle V(x)\xi,\xi\rangle \ge v|\xi|^2$, for every $x,\xi \in \Rd$ and some function $v\in C^{\alpha}_{\rm loc}(\Rd)$ bounded from below by a positive constant $c_0$;
\item
there exists a positive function $\varphi \in C^2(\Rd)$ blowing up as $|x|\to\infty$ such that $\mathcal{A}_v \varphi\le \lambda \varphi$ for some $\lambda>0$, where $\mathcal{A}_v= {\rm div}(QD)-vI$ $($see Hypothesis \ref{hyp_0}$(iv))$;
\item
the constants $\kappa, m$ and $p$ satisfy the condition $\kappa\le \sqrt{2 m^{-1}}$, if $p\in [2,\infty)$, and $p\kappa^2 m-4(p-1)^2\le 0$ if $p\in (1,2)$;
\item
if $p \in (1,2)$, then the function $v$ belongs to $C^2(\Rd)$, blows up at infinity and there exists a positive constant $M$ such that $(\mathcal{A}_v v)(x)\le M$ for every $x \in \Rd$.
\end{enumerate}
\end{hyp}

The assumptions on the matrix-valued function $Q$ and on $v$ guarantee that, for every $f\in C_b(\Rd)$, there exists a unique function $u\in C_b([0,\infty)\times\Rd)\cap C^{1,2}((0,\infty)\times\Rd)$ such that $D_tu={\mathcal A}_vu$ on $(0,\infty)\times\Rd$ and $u(0,\cdot)=f$ (see \cite{feller,newbook}). Setting
$S(t)f:=u(t,\cdot)$ for every $t\ge 0$, we define a semigroup of bounded operators on $C_b(\Rd)$ which satisfies the estimate
$\|S(t)\|_{{\mathcal L}(C_b(\Rd))}\le e^{-c_0t}$ (see, for instance, \cite[Proposition 2.2]{AngLor10Com}).

\begin{proposition}
\label{prop-3}
Fix $p\in (1,\infty)$ and $\f\in C_b(\Rd;\Rm)$. Then, under Hypotheses \ref{hyp_cb}, every locally in time bounded classical solution $\uu$ to the Cauchy problem
\begin{equation}\label{pb}
\left\{
\begin{array}{ll}
D_t \uu(t,x)= \A \uu(t,x),\qquad\;\, &(t,x)\in (0,\infty)\times \Rd,\\[1mm]
\uu(0,x)= \f(x), \;\, & x \in \Rd,
\end{array}
\right.
\end{equation}
satisfies the estimate
$|\uu(t,x)|^p \le (S(t)|\f|^p)(x)$ for every $(t,x)\in (0,\infty)\times \Rd$.
\end{proposition}

\begin{proof}
We first consider the case $p \in (1,2)$. A straightforward computation reveals that the function $z_{\varepsilon,p}=w_\varepsilon^p-S(\cdot)(|\f|^2+\varepsilon)^{\frac{p}{2}}$ is a classical solution to the equation
$D_t z_{\varepsilon,p}-\mathcal{A}_v z_{\varepsilon,p}=\psi_{\varepsilon,p}$
where
\begin{align*}
\psi_{\varepsilon,p}=&\frac{p(2-p)}{4}w_\varepsilon^{p-4}\q(|\uu|^2)-pw_\varepsilon^{p-2}\sum_{k=1}^m\q(u_k)\\
&+pw_\varepsilon^{p-2}\sum_{i=1}^d \langle B^iD_i\uu,\uu\rangle-pw_\varepsilon^{p-2}\langle V\uu,\uu\rangle+v w_\varepsilon^p.
\end{align*}
Using the hypotheses, Cauchy-Schwartz inequality which implies that $\q(|\uu|^2)\le 4w_\varepsilon^2 \sum_{k=1}^m\q(u_k)$ and Young's inequality, we infer that
\begin{align*}
\psi_{\varepsilon,p}\le p(1-p+\sigma)w_\varepsilon^{p-2}\sum_{k=1}^m\q(u_k)+ v w_\varepsilon^{p-2}\bigg [ |\uu|^2\bigg (1-p+\frac{p\kappa^2m}{4\sigma}\bigg )+\varepsilon\bigg ]
\end{align*}
for every $\sigma>0$.
Choosing $\sigma=p-1$ and taking Hypothesis \ref{hyp_cb}(iii) into account, we conclude that $D_tz_{\varepsilon,p}-\mathcal{A}_vz_{\varepsilon,p}\le \varepsilon v w_\varepsilon^{p-2}$, whence
\begin{align}
z_{\varepsilon,p}(t,x)\le \int_0^t (S(t-s)(\varepsilon v w_\varepsilon^{p-2}(s,\cdot)))(x)ds\le \varepsilon^{\frac{p}{2}}\int_0^t (S(t-s)v)(x)ds
\label{NUM}
\end{align}
for every $t>0$ and $x\in\Rd$, where we used the fact that, by Hypothesis \ref{hyp_cb}(iv) and \cite[(proof of) Lemma 4.1.3]{newbook}, we can apply each operator $S(t)$ to function $v$ and
$(S(t)v)(x)\le v(x)+Mt$ for every $t\ge 0$ and $x \in \Rd$. Letting $\varepsilon \to 0$ in \eqref{NUM}, the assertion follows in this case.

The case $p=2$ is simpler: it suffices to consider the function $z_{0,2}$ and observe that the function $\psi_{0,2}$ is nonpositive on $(0,\infty)\times\Rd$.
Thus, by applying a variant of the classical maximum principle (see \cite[Theorem 3.1.3]{newbook}) we deduce that $z(t,x) \le 0$ for every $(t,x)\in (0,\infty)\times \Rd$ and the assertion follows.

Finally, the case $p\in (2,\infty)$ follows easily from the case $p=2$ if we recall that $S(t)$ admits an integral representation with a kernel $p:(0,\infty)\times\Rd\times\Rd\to\R$  satisfying the condition $\|p(t,x,\cdot)\|_{L^1(\Rd)}\le 1$ for every $t>0$ and $x \in \Rd$ (see \cite[Theorem 1.2.5]{newbook}).
Hence, by the H\"older's inequality we conclude that $|\uu(t,\cdot)|^p\le (S(t)|\f|^2)^{\frac{p}{2}}\le S(t)|\f|^p$ in $\Rd$ for every $t>0$.
\end{proof}

The following theorem is an immediate consequence of Proposition \ref{prop-3}.
\begin{theorem}\label{exi_cb}
Under Hypotheses $\ref{hyp_cb}(i)$, $(ii)$, with $\kappa\le \sqrt{2m^{-1}}$, the Cauchy problem \eqref{pb}
admits a unique locally in time bounded classical solution $\uu$, for  every $\f\in C_b(\Rd;\Rm)$. Moreover,
\begin{equation}\label{est_inf}
|\uu(t,x)|\le e^{-c_0 t}\|\f\|_\infty, \qquad\;\, (t,x)\in [0,\infty) \times \Rd.
\end{equation}
\end{theorem}

\begin{proof}
The proof of this result is standard. The uniqueness of the locally in time bounded classical solution $\uu$ follows from Proposition \ref{prop-3}, with $p=2$.
The existence can be obtained by compactness, considering the sequence $(\uu_n)\subset C^{1+\alpha/2, 2+\alpha}_{\rm loc}((0,\infty)\times \Rd)\cap C_b([0,\infty)\times\Rd)$ of functions
such that $D_t\uu_n=\A\uu_n$ on $(0,\infty)\times B(0,n)$, $\uu_n$ vanishes on $(0,\infty)\times\partial B(0,n)$ and equals function $\f$ on $\{0\}\times B(0,n)$ for every $n\in\N$. Each function $\uu_n$ also satisfies the estimate $\|\uu_n(t,\cdot)\|_{\infty}\le e^{-c_0t}$ for every $t>0$. We refer the reader to \cite[Theorem 2.8]{AALT} for the missing details.
\end{proof}

Thanks to Theorem \ref{exi_cb} we can associate a semigroup ${\bm T}_{\infty}(t)$ to  $\A$ in $C_b(\Rd,\Rm)$, by setting ${\bm T}_\infty(\cdot)\f:=\uu$, where $\uu$ is the solution to the Cauchy problem
\eqref{pb} provided by Theorem \ref{exi_cb}. Clearly,
$\|{\bm T}_\infty(t)\|_{{\mathcal L}(C_b(\Rd;\Rm))}\le e^{-c_0t}$ for every $t>0$. This semigroup can be easily extended to $C_b(\Rd;\C^m)$ in a straightforward way.

In order to show that the semigroups $\T_p(t)$ ($p \ge p_0$) are consistent also with $\T_\infty(t)$ on $C_b(\Rd;\Rm)$ and, hence, on $C_b(\Rd;\C^m)$ we show that the domain $D(\A)$ of the weak generator of $\T_\infty(t)$ coincides with the maximal domain of $\A$ in $C_b(\Rd;\Rm)$.
The notion of weak generator ${\bm A}$ has been extended to vector-valued elliptic operators with unbounded coefficients in \cite{DelLor11OnA,AAL_Inv1} mimicking the classical definition of infinitesimal generator of a strongly continuous semigroup. Its domain is the set of all functions $\uu\in C_b(\Rd;\R^m)$ such that the function $t\mapsto t^{-1}(\T_{\infty}(t)\uu-\uu)$ is bounded in $(0,1]$ with values in $C_b(\Rd;\Rm)$ and it pointwise converges on $\Rd$ to a continuous function, which defines ${\boldsymbol A}\uu$.

\begin{proposition}\label{gen-Linfty}
Under Hypotheses $\ref{hyp_cb}(i)$, $(ii)$, with $\kappa\le \sqrt{2m^{-1}}$, the weak generator of the semigroup $\T(t)$ coincides with the realization of the operator $\A$ on $D_{\rm max}(\A)$, where
\begin{eqnarray*}
D_{\rm max}(\A)=\bigg\{\uu \in C_b(\Rd;\R^m)\cap\bigcap_{1\le p<\infty}W^{2,p}_{\rm loc}(\Rd,\R^m): \A\uu \in C_b(\Rd;\R^m)\bigg\}.
\end{eqnarray*}
Moreover, for every $\lambda>0$ and $\f\in C_b(\Rd;\Rm)$, the function $R(\lambda)\f$ defined by
\begin{equation}
(R(\lambda)\f)(x)=\int_0^{\infty}e^{-\lambda t}(\T_{\infty}(t)\f)(x)dt,\qquad\;\,x\in\Rd,
\label{R-lambda}
\end{equation}
belongs to $D_{\rm max}(\A)$ and solves the equation $\lambda\uu-\A\uu=\f$.
\end{proposition}

\begin{proof}
Taking \eqref{est_inf} into account and arguing as in the proof of \cite[Proposition 2.2]{AAL_Inv1}
the inclusion $D(\boldsymbol{A})\subset D_{\rm max}(\A)$ can be easily proved. Thus, $D(\boldsymbol{A})= D_{\rm max}(\A)$ if and only if $\lambda I-\boldsymbol A$ is injective on $D_{\rm max}(\A)$. To prove the injectivity of this operator, we fix a function $\uu\in D_{\rm max}(\A)$ such that $\lambda\uu-\A\uu={\bm 0}$. A straightforward computation reveals that $2\lambda |\uu|^2-{\mathcal A}_v|\uu|^2=\psi$, where
\begin{eqnarray*}
\psi=-2\sum_{k=1}^m\q(u_k)+2\sum_{i=1}^d\langle B^iD_i\uu, \uu\rangle-2\langle V\uu,\uu\rangle+v|\uu|^2.
\end{eqnarray*}
Using Hypotheses \ref{hyp_cb}(i), (ii), the condition $\kappa\le \sqrt{2m^{-1}}$ and Young's inequality to estimate
$\sum_{i=1}^d\langle B^iD_i\uu, \uu\rangle$ $\le \sum_{k=1}^m\q(u_k)+\frac{\kappa m^2}{4}v|\uu|^2$, we conclude first that $\psi\le 0$ on $\Rd$ and then,
applying the maximum principle in \cite[Theorem 3.1.6]{newbook}, that $\uu=\bm 0$.
\end{proof}

\begin{proposition}
\label{prop-5}
Suppose that the assumptions of Theorem $\ref{mainth}$ are satisfied for every $p\ge p_0$ and some $p_0>1$. Further assume that the functions $q_{ij}$ belongs to $C^{1+\alpha}_{\rm loc}(\Rd)$ $(i,j=1,\ldots,d)$ and
$\langle Q(x)\xi,\xi\rangle\ge\mu_0|\xi|^2$ for every $x,\xi\in\Rd$ and some positive constant $\mu_0>0$. Finally, suppose that Hypotheses $\ref{hyp_cb}(ii)$-$(iv)$ are satisfied. Then, for every $\bm f\in C_b(\Rd;\C^m)\cap L^p(\Rd,\C^m)$, $p\ge p_0$ and $t>0$ it holds that ${\bm T}_p(t)\f={\bm T}_{\infty}(t)\f$.
\end{proposition}

\begin{proof}
As we have already stressed, we can limit ourselves to proving that $\T_p(t)\f=\T_q(t)\f$ for every $\f\in C_b(\Rd;\Rm)\cap L^p(\Rd;\Rm)$.
We fix $\lambda>0$, $\f\in C_b(\Rd;\Rm)$, with compact support, and consider the resolvent equation $\lambda\uu-\A\uu=\f$.
By Theorem \ref{mainth} and Proposition \ref{gen-Linfty} the previous equation admits a unique solution $\uu_r\in D({\bm A}_r)$, for every $r\in [p_0,\infty)$, and a unique solution ${\bf u}_{\infty}\in D_{\max}(\A)$.
Since $\uu_r$ is the Laplace transform of the semigroup $\T_r(t)$, the consistency of the semigroups ${\bm T}_p(t)$ and ${\bm T}_q(t)$
implies that $\uu_p=\uu_q$ for $p,q\in [p_0,\infty)$. Hence, we can simply write $\uu$ instead of $\uu_p$, and $\uu\in W^{2,q}_{\rm loc}(\Rd;\Rm)$ for every $q\in [p_0,\infty)$, so that it is continuous over $\Rd$. To prove that it is also bounded on $\Rd$, we fix $r>0$, arbitrarily, and $R>0$ such that ${\rm supp}(\f)\subset B(0,R)$. Since the operator ${\bm A}_q$ is dissipative,
we can write
$\|\uu\|_{L^q(B(0,r);\Rm)}\le\|\uu\|_q\le\lambda^{-1}\|\f\|_{L^q(B(0,R);\Rm)}$.
Letting $q$ tend to $\infty$ and using the arbitrariness of $r>0$, we conclude that $\uu$ is bounded over $\Rd$, so that it belongs to
$C_b(\Rd;\Rm)$. Since $\A u=\lambda\uu-\f$, the function $\A u$ is bounded and continuous over $\Rd$. Thus, $\uu\in D_{\max}(\A)$ and the equality $\uu=\uu_{\infty}$ follows.

Next, we fix a function $\bm\psi\in C^{\infty}_c(\Rd;\Rm)$, write $\langle \uu,\bm\psi\rangle=
\langle\uu_{\infty},\bm\psi\rangle$ and integrate both sides of this equality over $\Rd$. Taking \eqref{R-lambda} into account,
we deduce that
\begin{align*}
\int_{\Rd}\uu\bm\psi dx=&\int_0^{\infty}e^{-\lambda t}dt\int_{\Rd}\bm\psi \T_p(t)\f dx\\
=&\int_0^{\infty}e^{-\lambda t}dt\int_{\Rd}\bm\psi\T_{\infty}(t)\f dx =\int_{\Rd}\uu_{\infty}\bm\psi dx.
\end{align*}
From the uniqueness of the Laplace transform, we conclude that
\begin{eqnarray*}
\int_{\Rd}\bm\psi\T_p(t)\f dx=\int_{\Rd}\bm\psi\T_{\infty}(t)\f dx,
\end{eqnarray*}
from which the equality $\T_p(t)\f=\T_{\infty}(t)\f$ follows for every $t>0$. To remove the condition on the support of $\f$,  we observe that
every function $\f\in C_b(\Rd;\Rm)\cap L^p(\Rd;\Rm)$ is the limit in $L^p(\Rd;\Rm)$ of a sequence $(\f_n)\subset C^{\infty}_c(\Rd;\Rm)$ which is bounded in $C_b(\Rd;\Rm)$ and converge to $\f$ pointwise in $\Rd$. Taking the limit as $n$ tends to $\infty$ in the equality
$\T_p(t)\f_n=\T_{\infty}(t)\f_n$, we complete the proof.
\end{proof}

\begin{theorem}
\label{thm_hyp}
Let the assumptions of Proposition $\ref{prop-5}$ be satisfied.
Then, the semigroup $\T(t)$ maps $L^p(\Rd;\C^m)$ into $L^q(\Rd;\C^m)$ for every $t>0$ and $p_0\le p\le q\le\infty$, and, for every $T>0$, there exists a positive constant $C=C(T,p,q)$ such that
\begin{equation}\label{hyp_est}
\|\T(t)\f\|_{L^q(\Rd;\C^m)}\le Ct^{-\frac{d}{2}\left (\frac{1}{p}-\frac{1}{q}\right )}\|\f\|_{L^p(\Rd;\C^m)}, \qquad\;\, \f \in L^p(\Rd;\C^m),\;\,t\in (0,T].
\end{equation}
\end{theorem}

\begin{proof}
Also in this case we provide the assertion for functions with values in $\R^m$.
For every $p\in [p_0,\infty)$, we consider the contraction semigroup $S_p(t)$ generated by
the operator $\mathcal  A_v$, with domain $D_p(\mathcal A_v):=\{u\in L^p(\Rd)\,\mid\, {\mathcal A}_0u, vu\in L^p(\Rd)\}$,
in $L^p(\Rd)$ (see \cite[Theorem 2.4]{mpps}). The restriction of the semigroup $S_2(t)$ to $L^2(\Rd)\cap L^1(\Rd)$ can be extended to a contraction $C_0$-semigroup $S_1(t)$ on $L^1(\R^d)$ which is consistent with $S_p(t)$ for each $p\in (1,\infty)$.
Indeed, fix $f\in L^2(\Rd)\cap L^1(\Rd)$ and $r>0$. Then, $f\in L^p(\Rd)$ for every $1<p<2$ and, since the semigroups
$S_p(t)$ are consistent, we get
\begin{eqnarray*}
\|S_2(t)f\|_{L^1(B(0,r))}=\lim_{p\to 1^+}\|S_p(t)f\|_{L^p(B(0,r))} \leq \limsup_{p\to 1}\|f\|_{L^p(\Rd)} \leq\|f\|_{L^1(\Rd)}.
\end{eqnarray*}
By letting $r$ tend to $\infty$, we conclude that
the restriction of $S_2(t)$ to $L^2(\Rd)\cap L^1(\Rd)$ extends, by density, to a contraction semigroup $S_1(t)$ on $L^1(\Rd)$.
Moreover, if $f\in C_c^\infty(\Rd)$, then
\begin{eqnarray*}
S_1(t)f-f =S_p(t)f-f=\int_0^t S_p(s)\mathcal A_v f ds,\qquad\;\,t>0.
\end{eqnarray*}
Hence for every $p>1$ it follows that
$\|S_1(t)f-f\|_{L^p(B(0,r))} \leq t\|\mathcal A_vf\|_p$.
Letting first $p$ tend to $1$ and then $r$ tend to $\infty$, we deduce that
$S_1(t)f$ converges to $f$ in $L^1(\Rd)$ as $t$ tends to $0$.
By density it follows that $S_1(t)$ is a strongly continuous semigroup.

On the other hand, $S_2(t)$ is the semigroup associated with the quadratic form
\begin{eqnarray*}
a(f)=\int_{\Rd}\q(f)dx +\int_{\Rd} v|f|^2dx
\end{eqnarray*}
with domain $D(a)=\{f\in L^2(\Rd)\cap W_{\rm loc}^{1,2}(\Rd): |Q^{1/2}\nabla f|, v^{1/2}f\in L^2(\Rd)\}$.
Note that $a(f)\geq \min\{c_0,\mu_0\}\|f\|_{1,2}$ for every $f\in D(a)$.
Since by Nash's inequality
$\|f\|_2^{2+4/d} \leq C_1\|f\|_{1,2}^2\|f\|_1^{4/d}$
for some positive constant $C_1$ and every $f\in W^{1,2}(\Rd)\cap L^1(\Rd)$,
we easily obtain that $\|f\|_2^{2+4/d}\leq K a(f)\|f\|_1^{4/d}$
for some constant $K>0$ , and therefore $S_1(t)$ is bounded from $L^1(\Rd)$ into $L^2(\Rd)$  (see \cite[Theorem 2.4.6]{davies}).
By observing that $\mathcal A_v$ is self-adjoint, the usual duality argument proves that
$S_1(t)$ is bounded from $L^2(\Rd)$ into $L^\infty(\Rd)$ and, by applying the semigroup law, it follows that
$S_1(t)$ is bounded from $L^1(\Rd)$ into $L^\infty(\Rd)$.

Using this property together with the estimate in the statement of Proposition \ref{prop-3} it is immediate to check that $\T_p(t)$ maps $L^p(\Rd;\Rm)$ into
$L^{\infty}(\Rd;\Rm)$ for every $p\ge p_0$ and \eqref{hyp_est} follows in this case.

Finally,
applying Riesz-Thorin interpolation theorem, we conclude the proof.
\end{proof}

\begin{remark}
\label{rmk-final}
{\rm It is worth noticing that, if the assumptions of Proposition $\ref{prop-5}$ are satisfied, with \eqref{drift-cont} replaced by \eqref{new-k}, to be satisfied by every $\kappa>0$, and $\theta<1$, then the semigroup $\T_p(t)$ is defined for every $p>1$ (since, by Remark \ref{forallp}, the assumptions of Theorem \ref{mainth} are satisfied by every $p>1$) and estimate \eqref{hyp_est} holds true for every $p,q\in (1,\infty)$, with $p<q$.}
\end{remark}

\begin{example}
{\rm Let $\A$ be the operator defined in Example \ref{example-1} with a potential matrix $V$ whose entries belong to $C^\alpha_{\rm loc}(\Rd)$ and satisfy \eqref{exp-V}. Assuming further that $A_0\le \sqrt{2m^{-1}}$ if $p \ge 2$ and $pA_0^2m-4(p-1)^2\le 0$ if $p \in (1,2)$, all the assumptions in Hypotheses \ref{hyp_cb} are satisfied. Hypothesis \ref{hyp_cb}(ii) is satisfied, for instance, by the function $\varphi:\Rd\to\R$, defined by $\varphi(x)=1+|x|^2$ for every $x \in \Rd$.
Indeed, since
\begin{eqnarray*}
(\mathcal{A}_v \varphi)(x)= 2\alpha |x|^2(1+|x|^2)^{\frac{\alpha}{2}-1}+ 2d(1+|x|^2)^{\frac{\alpha}{2}}-(1+|x|^2)e^{(1+|x|^2)^\beta}
\end{eqnarray*}
diverges to $-\infty$ as $|x|\to\infty$, clearly we can find a positive constant $\lambda$ such that $\mathcal{A}_v \varphi\le \lambda \varphi$. Further,
\begin{align*}
(\mathcal{A}_v v)(x)= v(x)\big [&2\beta(\alpha+2\beta-2)|x|^2(1+|x|^2)^{\frac{\alpha}{2}+\beta-2}+4\beta^2|x|^2(1+|x|^2)^{\frac{\alpha}{2}+2\beta-2}\\
&+2\beta d(1+|x|^2)^{\frac{\alpha}{2}+\beta-1}-e^{(1+|x|^2)^\beta}\big ]=:v(x)\psi(x)
\end{align*}
for every $x \in \Rd$. Since $\psi(x)$ diverges to $-\infty$ as $|x|\to\infty$,  Hypothesis \ref{hyp_cb}(iv) is satisfied too. Thus, Proposition \ref{prop-5} can be applied to deduce that $\T_p(t)\f=\T_{\infty}(t)\f$ for every $\f\in L^p(\Rd;\C^m)\cap C_b(\Rd;\C^m)$.

Suppose that $(2\gamma+\beta)|\lambda| c_0<1$ (see Example \ref{example-1}) and $\gamma<\alpha/4$. Then the semigroup $\T_p(t)$ exists for every $p\in (1,\infty)$ and all the semigroups are consistent.

Finally, Remark \ref{rmk-final} infers that $\T(t)$ maps $L^p(\Rd;\C^m)$ into $L^q(\Rd;\C^m)$ for every $t>0$ and $1< p\le q\le\infty$ and estimate \eqref{hyp_est} holds true.}
\end{example}

\appendix

\section{}

\small

\subsection{Derivation of condition \eqref{pdis}}
\label{app-a.1}

We need to compute the supremum of the function $f_2$ in \eqref{funct-f2} in the set $\Omega=\{(\varepsilon_1,\varepsilon_2,\varepsilon_3,\varepsilon_4)\in (0,\infty)^4: g_2(\varepsilon_1,\varepsilon_2,\varepsilon_3,\varepsilon_4)>0\}$, where $g_2$ is defined by \eqref{funct-g2}.
To simplify the notation, we write $f_2(\varepsilon_1,\varepsilon_2,\varepsilon_3,\varepsilon_4)=A-B\varepsilon_1^{-1}-C\varepsilon_2^{-1}-D\varepsilon_3^{-1}-D\varepsilon_4^{-1}$ and $g_2(\varepsilon_1,\varepsilon_2,\varepsilon_3,\varepsilon_4)=E-\varepsilon_1-F\varepsilon_2-G\varepsilon_3-H\varepsilon_4$, where
\begin{align*}
\begin{array}{lllll}
&A=1-\frac{\theta}{p}-\frac{p-1}{p}\kappa m\gamma,\qquad\;\, &B=\frac{\gamma^2}{4},\qquad\;\, &C=\gamma^2\frac{|p-2|}{4},\qquad\;\,&D=\frac{\kappa m|p-2|}{4p}, \\[1mm]
&E=p-1,\qquad\;\,& F=2-p,\qquad\;\,&G=\kappa\frac{2-p}{p},\qquad\;\,&H=\frac{2-p}{p}\kappa m.
\end{array}
\end{align*}
From the constrain $g_2(\varepsilon_1,\varepsilon_2,\varepsilon_3,\varepsilon_4)>0$, we deduce that $\varepsilon_1<E-F\varepsilon_2-G\varepsilon_3-H\varepsilon_4$. Therefore, taking also into account that $\varepsilon_1>0$, we deduce that
\begin{eqnarray*}
f_2(\varepsilon_1,\varepsilon_2,\varepsilon_3,\varepsilon_4)<\widetilde f_2(\varepsilon_2,\varepsilon_3,\varepsilon_4):=A-\frac{B}{E-F\varepsilon_2-G\varepsilon_3-H\varepsilon_4}-\frac{C}{\varepsilon_2}-\frac{D}{\varepsilon_3}-\frac{D}{\varepsilon_4}
\end{eqnarray*}
for every $(\varepsilon_1,\varepsilon_2,\varepsilon_3,\varepsilon_4)\in\Omega$.
So, we consider the function $\widetilde f_2$ on the set $\widetilde\Omega=\{(\varepsilon_2,\varepsilon_3,\varepsilon_4)\in (0,\infty)^3: F\varepsilon_2+G\varepsilon_3+H\varepsilon_4<E\}$. Since $\widetilde f_2(\varepsilon_2,\varepsilon_3,\varepsilon_4)$ diverges to $-\infty$ as $(\varepsilon_2,\varepsilon_3,\varepsilon_4)$ approaches
the points of the boundary of $\widetilde\Omega$, this function has a maximum in $\widetilde\Omega$. Imposing that its gradient vanishes, we obtain the system
\begin{eqnarray*}
\left\{
\begin{array}{l}
\varepsilon_2=(E-F\varepsilon_2-G\varepsilon_3-H\varepsilon_4)\sqrt{\frac{C}{BF}},\\[2mm]
\varepsilon_3=(E-F\varepsilon_2-G\varepsilon_3-H\varepsilon_4)\sqrt{\frac{D}{BG}},\\[2mm]
\varepsilon_4=(E-F\varepsilon_2-G\varepsilon_3-H\varepsilon_4)\sqrt{\frac{D}{BH}}
\end{array}
\right.
\Longleftrightarrow\left\{
\begin{array}{l}
\varepsilon_2=(E-F\varepsilon_2-G\varepsilon_3-H\varepsilon_4)\sqrt{\frac{C}{BF}},\\[2mm]
\varepsilon_3=\sqrt{\frac{DF}{CG}}\varepsilon_2,\\[2mm]
\varepsilon_4=\sqrt{\frac{DF}{CH}}\varepsilon_2.
\end{array}
\right.
\end{eqnarray*}
We conclude that the unique stationary point of $\widetilde f_2$ is the point $(\overline{\varepsilon}_2,\overline{\varepsilon}_3,\overline{\varepsilon}_4)$, where
\begin{align*}
&\overline{\varepsilon}_2=\frac{E\sqrt{\frac{C}{F}}}{\sqrt{B}+\sqrt{CF}+\sqrt{DG}+\sqrt{DH}},
\qquad\;\,\overline{\varepsilon}_3=\frac{E\sqrt{\frac{D}{G}}}{\sqrt{B}+\sqrt{CF}+\sqrt{DG}+\sqrt{DH}},\\
&\overline{\varepsilon}_4=\frac{E\sqrt{\frac{D}{H}}}{\sqrt{B}+\sqrt{CF}+\sqrt{DG}+\sqrt{DH}},
\end{align*}
i.e.,
\begin{align}
\widetilde f_2(\overline{\varepsilon}_2,\overline{\varepsilon}_3,\overline{\varepsilon}_4)=&A-\frac{1}{E}(\sqrt{B}+\sqrt{CF}+\sqrt{DG}+\sqrt{DH})^2,
\label{corona}
\end{align}
which is the left-hand side of \eqref{pdis}.
Since $f_2$ is continuous at the point $(\overline \varepsilon_1,\overline{\varepsilon}_2,\overline{\varepsilon}_3,\overline{\varepsilon}_4)$, where $\overline \varepsilon_1=E-F\overline{\varepsilon}_2-G\overline{\varepsilon}_3-H\overline{\varepsilon}_4$, and
$\widetilde f_2(\overline{\varepsilon}_2,\overline{\varepsilon}_3,\overline{\varepsilon}_4)=\lim_{\varepsilon_1\to \overline{\varepsilon}_1}f_2(\varepsilon_1,\overline{\varepsilon}_2,\overline{\varepsilon}_3,\overline{\varepsilon}_4)$,
the second line of \eqref{corona} is also the supremum in $\Omega$ of function $f_2$.

\subsection{Derivation of condition \eqref{pdis-1}}
\label{app-a.2}

We still need to consider function $f_2$ but in the new set $\Omega=\{(\varepsilon_1,\varepsilon_2,\varepsilon_3,\varepsilon_4)\in (0,\infty)^4: E-E\varepsilon_1-F\varepsilon_3>0,\ E-E\varepsilon_2-G\varepsilon_4>0\}$, where now
$E=p$, $F=\kappa (p-2)$ and $G=\kappa m$.
Since $0<\varepsilon_1<1-E^{-1}F\varepsilon_3$ and $0<\varepsilon_2<1-E^{-1}G\varepsilon_4$, it follows that
\begin{eqnarray*}
f_2(\varepsilon_1,\varepsilon_2,\varepsilon_3,\varepsilon_4)<A-\frac{BE}{E-F\varepsilon_3}-\frac{CE}{E-G\varepsilon_4}-\frac{D}{\varepsilon_3}-\frac{D}{\varepsilon_4}=:\widehat f_2(\varepsilon_3,\varepsilon_4)
\end{eqnarray*}
for every $(\varepsilon_1,\varepsilon_2,\varepsilon_3,\varepsilon_4)\in\Omega$. We consider function $\widehat f_2$ in $\Omega'=(0,F^{-1}E)\times (0,G^{-1}E)$, which diverges to $-\infty$ as $(\varepsilon_3,\varepsilon_4)$ approaches the boundary of $\Omega'$. Hence, it has a maximum in $\Omega'$, which is attained at the unique stationary point $(\overline{\varepsilon}_3,\overline{\varepsilon}_4)$, where
\begin{eqnarray*}
\overline{\varepsilon}_3=\sqrt{\frac{D}{F}}\frac{E}{\sqrt{BE}+\sqrt{DF}},\qquad\;\,\overline{\varepsilon}_4=\sqrt{\frac{D}{G}}\frac{E}{\sqrt{CE}+\sqrt{DG}},
\end{eqnarray*}
and
\begin{align}
\widehat f_2(\overline \varepsilon_3,\overline \varepsilon_4)=&A-B-2\sqrt{\frac{BDF}{E}}-\frac{DF}{E}-C-2\sqrt{\frac{CDG}{E}}-\frac{DG}{E},
\label{corona-2}
\end{align}
which is the left-hand side of \eqref{pdis-1}.
Since $f_2$ is continuous at the point $(\overline \varepsilon_1,\overline \varepsilon_2,\overline \varepsilon_3,\overline \varepsilon_4)$, where $\overline \varepsilon_1=1-E^{-1}F\overline \varepsilon_3$, $\overline \varepsilon_2=1-E^{-1}G\overline \varepsilon_4$ and
$\widehat f_2(\overline{\varepsilon}_3,\overline{\varepsilon}_4)=\lim_{(\varepsilon_1,\varepsilon_2)\to (\overline \varepsilon_1,\overline \varepsilon_2)}f_2(\varepsilon_1,\varepsilon_2,\overline{\varepsilon}_3,\overline{\varepsilon}_4)$, we conclude that
the second line of \eqref{corona-2} is also the supremum in $\Omega$ of function $f_2$.

\subsection{On the inequality \eqref{last}}
\label{app-a.3}
In this subsection, we prove in details that we can choose the parameters $\varepsilon_i$ ($i=0,\ldots,3$) in such a way that the first three terms in the right-hand side of \eqref{last} are all nonpositive.
Equivalently, we just need to prove that these parameters can be fixed to guarantee that $f_1(\varepsilon_1,\varepsilon_3)>0$, $1-\varepsilon_1\kappa (p-2)p^{-1}\ge 0$ and $p-\kappa m\varepsilon_3>0$.
The better choice of $\varepsilon_1$ is $\varepsilon_1=p[\kappa(p-2)]^{-1}$. Replacing this value in $f_1$, we get a new function $\widetilde f_1$, which is continuous and increasing in the domain $(0,p(\kappa m)^{-1})$.
Therefore its supremum in the interval $(0,p(\kappa m)^{-1})$ is the limit at $p(\kappa m)^{-1}$, which is
$1-\frac{\theta}{p}-\frac{\kappa^2m(p-2)}{4p^2}(m+p-2)$, and it is positive due to condition \eqref{pdis-1}.

\subsection{On the inequality \eqref{last-1}}
\label{app-a.4}
Here, we prove in details that the free parameters in \eqref{last-1} can be chosen to guarantee that the first two terms in the right-hand side of \eqref{last-1} are all nonpositive.
Of course, we just need to show that we can choose the positive parameters $\varepsilon_1$ and $\varepsilon_3$ in such a way that $f_2(\varepsilon_1,\varepsilon_3)$ and $g_1   (\varepsilon_1,\varepsilon_3)$ are both nonnegative.
(Here, $f_2$ is the same function defined in \eqref{funct-f2}, where now $\gamma=0$, and $g_1$ is defined by \eqref{funct-g1}.)
For this purpose, we compute the supremum of the function $f_2$ on the set $\Omega=\{(\varepsilon_1,\varepsilon_3)\in (0,\infty)^2: g_1(\varepsilon_1,\varepsilon_3)>0\}$. We set $E=m$, $G=p(p-1)[(2-p)\kappa]^{-1}$, so that $g_1(\varepsilon_1,\varepsilon_3)>0$  is equivalent to $\varepsilon_1+F\varepsilon_3<G$. Therefore, $0<\varepsilon_1<G-F\varepsilon_3$ and
\begin{eqnarray*}
f_2(\varepsilon_1,\varepsilon_3)\le A'-\frac{B'}{\varepsilon_3}-\frac{B'}{C'-m\varepsilon_3}=:\overline f_2(\varepsilon_3)
\end{eqnarray*}
for every $(\varepsilon_1,\varepsilon_3)\in\Omega$, where
\begin{eqnarray*}
A'=1-\frac{\theta}{p},\qquad\;\,B'=\frac{\kappa m(2-p)}{4p},\qquad\;\, C'-\frac{p(p-1)}{(2-p)\kappa}.
\end{eqnarray*}
We consider function $\overline f_2$ in the interval $(0,m^{-1}C)$. Since it diverges to $-\infty$ as $\varepsilon_3$ tends to $0$ and $m^{-1}C$, function $\overline f_2$ has a maximum. Computing the first-order derivative
of $\overline f_2$, we easily realize that such a maximum is attained at $\overline \varepsilon_3=c(m+\sqrt{m})^{-1}$ and its value is
\begin{eqnarray*}
A'-\frac{B'}{C'}(\sqrt{m}+1)^2=1-\frac{\theta}{p}-\frac{\kappa^2(2-p)^2}{4p^2(p-1)}(\sqrt{m}+m)^2,
\end{eqnarray*}
which is positive due to condition \eqref{pdis}.

If we set $\overline\varepsilon_1=G-F\overline\varepsilon_3$ and observe that $f_2$ is continuous at $(\overline\varepsilon_1,\overline\varepsilon_3)$, then we conclude that $\lim_{\varepsilon_1\to\overline \varepsilon_1}f_2(\varepsilon_1,\overline \varepsilon_3)=\overline f_2(\overline\varepsilon_3)$.
Moreover, $g_1(\overline \varepsilon_1,\overline \varepsilon_3)=0$. These two remarks show that we can choose $\varepsilon_1, \varepsilon_3\in (0,\infty)$ sufficiently small such that $f_2(\varepsilon_1,\varepsilon_3)>0$
and $g_1(\varepsilon_1,\varepsilon_3)>0$.

\end{document}